\documentclass[11pt]{article}

\usepackage{pgfplots,tikz,url}

\usepackage{fullpage}

\usepackage{amsmath,amssymb,amsthm}
\usepackage[capitalize]{cleveref}
\usepackage{mathtools}
\usepackage{graphicx}
\usepackage{enumitem}

\usepackage{palatino}
\usepackage{mathpazo}
\usepackage{fullpage}
\usepackage[margin=1cm]{caption}

\usepackage{thmtools}
\usepackage{thm-restate}

\newcommand{\ignore}[1]{}
\newcommand{\N}{\mathbb{N}}

\newcommand{\R}{\mathbb{R}}

\newcommand{\mQ}{\operatorname{\mathcal Q}}

\renewcommand{\sup}{\mathrm{sup}}
\renewcommand{\min}{\mathrm{min}}

\renewcommand{\max}{\mathrm{max}}
\renewcommand{\epsilon}{\varepsilon}
\newcommand{\eps}{\epsilon}

\def\01{\{0,1\}}

\newtheorem{defin}{Definition} 
\newtheorem{definition}[defin]{Definition}
\newtheorem{proposition}[defin]{Proposition}
\newtheorem{theorem}{Theorem}
\newtheorem*{theorem*}{Theorem}

\newtheorem{corollary}[defin]{Corollary}
\newtheorem{lemma}[defin]{Lemma}

\newtheorem*{claim*}{Claim}

\newtheorem*{conjecture*}{Conjecture}
\theoremstyle{definition}

\newcommand{\const}[1]{1/(10 #1^2)}

\DeclareMathOperator{\arccosh}{arccosh}

\newcommand{\mon}{{1,\mathrm{mon}}}
\newcommand{\cheb}{{1,\mathrm{cheb}}}

\newcommand{\funcmin}{f_{\min}}
\newcommand{\funcmax}{f_{\max}}

\newcommand{\Chebyconst}{\mathfrak{c}}

\def\keywords{\vspace{.5em}
{\textit{Keywords}:\,\relax%
}}

\date{\today}
\begin{document}

\title{Revisiting the convergence rate of the Lasserre hierarchy for polynomial optimization over the hypercube}
\author{Sander Gribling\thanks{S.J.Gribling@tilburguniversity.edu} \and Etienne de Klerk\thanks{E.deKlerk@tilburguniversity.edu} \and Juan Vera\thanks{J.C.VeraLizcano@tilburguniversity.edu} }
\date{Department of Econometrics and Operations Research, Tilburg University, The Netherlands}

\maketitle

\begin{abstract}
We revisit the problem of minimizing a given polynomial $f$ on the hypercube $[-1,1]^n$. Lasserre's hierarchy
(also known as the moment- or sum-of-squares hierarchy) provides a sequence of lower bounds $\{f_{(r)}\}_{r \in \mathbb N}$
 on the minimum value $f^*$, where $r$ refers to the allowed degrees in the sum-of-squares hierarchy. A natural question
  is how fast the hierarchy converges as a function of the parameter $r$. The current state-of-the-art is due to Baldi
   and Slot [SIAM J. on Applied Algebraic Geometry, 2024] and roughly shows a convergence rate of order $1/r$.
    Here we obtain closely related results via a different approach: the polynomial kernel method. We
     also discuss limitations of the polynomial kernel method, suggesting a lower bound of order $1/r^2$ for our approach.

\end{abstract}


\section{Introduction}
We consider the problem of minimizing a given polynomial $f:\mathbb{R}^n \rightarrow \mathbb{R}$ on the hypercube $[-1,1]^n$, namely
\begin{equation}
\label{prob:origin}
f_\min := \min_{x \in [-1,1]^n} f(x).
\end{equation}
For given $\varepsilon \in [0,1]$, we say that a value $\hat f$ is a (relative) $\varepsilon$-approximation of $f_\min$ if
\[
|\hat f - f_\min| \le \varepsilon (f_\max - f_\min),
\]
where $f_\max := \max_{x \in [-1,1]^n} f(x)$; see e.g.\ \cite{10.5555/3226669.3227146}. Since we will only consider polynomials, we write $f \in \R[x]$ to indicate a real polynomial, and $f \in \R[x]_d$ if the degree of $f$ is (at most) $d$.

Problem \eqref{prob:origin} is of enduring interest in the context of bound-constrained nonlinear optimization, and includes
the $\mathsf{NP}$-hard maximum cut problem as a special case: for a graph $G = (V,E)$ with $|V| = n$, the size of a maximum cut in $G$ is given by
\[
\max_{x \in [-1,1]^n} \frac{1}{4}x^\top L_G x,
\]
where $L_G$ is the Laplacian matrix of $G$. While the maximum cut problem allows a  polynomial-time $(1-0.878)$-approximation,
due to Goemans and Williamson \cite{10.1145/227683.227684},
 Bellare and Rogaway \cite{10.5555/3226669.3227146}
showed that, in general, problem \eqref{prob:origin} does not allow a polynomial-time
 $\varepsilon$-approximation of $f_\min$ for any $\varepsilon \in [0,1)$, unless $\mathsf{P}=\mathsf{NP}$.

Problem \eqref{prob:origin} is a special case of a polynomial optimization problem with a compact feasible set.
In a seminal work from 2001, Lasserre \cite{doi:10.1137/S1052623400366802} introduced a hierarchy of semidefinite programming problems to solve such problems.
To introduce the Lasserre hierarchy for our special case, we view $[-1,1]^n$ as the feasible set for constraints
\begin{equation}\label{g cube}
  g_i(x) := 1-x_i^2 \ge 0 \quad(i \in \{1,\ldots,n\} =: [n]).
\end{equation}
 The quadratic module associated with this (or any) set of polynomials $\mathbf{g} = (g_1,\ldots,g_n)$ is defined as
\begin{equation}\label{eq:quad mod}
  \mathcal{Q}(\mathbf{g}) = \Sigma[x] + \sum_{i=1}^n g_i \Sigma[x],
\end{equation}
where $\Sigma[x]$ denotes the cone of sum-of-squares (SOS) polynomials in the variables $x$.
The truncated, quadratic module of order $r$ is defined as
\[
\mathcal{Q}(\mathbf{g})_r = \Sigma[x]_r + \sum_{i=1}^n g_i \Sigma[x]_{r - \deg(g_i)},
\]
where $\Sigma[x]_r$ denotes the SOS cone of polynomials of (total) degree at most $r$.
The key idea is that membership of $\mathcal{Q}(\mathbf{g})_r$ may be formulated as a semidefinite programming~(SDP) problem.
The Lasserre hierachy may now be defined as the sequence of SDP problems:
\begin{equation}
\label{eq:lasserre hierarchy}
  f_{(r)} := \sup \left\{ t \; : \; f - t \in \mathcal{Q}(\mathbf{g})_r\right\} \quad r \in \N.
\end{equation}
The convergence of the sequence $f_{(r)} \rightarrow f_\min$ as $r \rightarrow \infty$ was already shown
by Lasserre \cite{doi:10.1137/S1052623400366802} for the case when the quadratic module satisfies the Archimedean assumption:
\begin{equation}
\label{Archimedean_property}
\exists g \in \mathcal{Q}(\mathbf{g}) \; \mbox{such that} \; \{x \in \R^n \; : \; g(x) \ge 0\} \mbox{ is compact.}
\end{equation}
 This is an immediate consequence
of the following Positivstellensatz by Putinar \cite{Putinar1993}.

\begin{theorem}[Putinar Positivstellensatz \cite{Putinar1993}]
\label{thm:Putinar}
Assume a compact semi-algebraic set $K$ is described by polynomials $\mathbf{g}$ and that the  Archimedean assumption \eqref{Archimedean_property} holds for $\mathcal{Q}(\mathbf{g})$.
Then, if $f \in \R[x]$ is (strictly) positive on $K$, one has $f \in \mathcal{Q}(\mathbf{g})$.
\end{theorem}
The Archimedean property clearly holds for the polynomials that we use to describe the hypercube. Since we are only interested in this special case,
$\mathbf{g}$ will always refer to \eqref{g cube} from now on.

Following the seminal paper by Lasserre, there has been extensive research on cases where finite convergence holds \cite{Nie2012OptimalityCA,doi:10.1137/080728214,doi:10.1137/100814147}, as well as research into the
 rate of convergence \cite{NIE2007135,10.1007/s10107-020-01537-7,doi:10.1137/23M1555430}.
For the case of the hypercube $[-1,1]^n$, the best result so far is due to Baldi and Slot \cite{doi:10.1137/23M1555430}, who showed that
\[
f_{(r)} - f_\min = O\left(\frac{1}{r}\right).
\]
Here the big-O notation suppresses all dependence on $f$ and $n$.

To understand this result better, we consider  the main result of Baldi and Slot \cite{doi:10.1137/23M1555430} in more detail.

\begin{theorem}[Theorem 11 in Baldi and Slot \cite{doi:10.1137/23M1555430}]
\label{thm:Baldi_Slot}
 Let $f \in \R[x]$ be a polynomial of degree $d$ and assume  $0 < \funcmin < \funcmax$ on $[-1,1]^n$. Then we
 have $f \in \mathcal{Q}(\mathbf{g})_{rn}$, with $\mathbf{g}$ defined by \eqref{g cube}, whenever
\[
    r \geq
    4 \Chebyconst \cdot d^2(\log n) \cdot \frac{\funcmax}{\funcmin} + \max \left\{ \pi d \sqrt{2n}, \ \left(2 \Chebyconst \cdot \frac{\funcmax}{\funcmin} \cdot C(n, d)\right)^{1/2} \right\},
\]
where $\Chebyconst \in [1,e^5]$ is an absolute constant with $e \approx 2.71828$ being Euler's number, and $C(n, d)$ is a constant that only depends on $(n,d)$, and satisfies\footnote{The bounds given here on the constant $C(n,d)$ are from Laurent and Slot \cite[Section 3.3]{Laurent2021AnEV}.}
\[
2\pi^2d^2n \binom{n+d}{d} \le C(n, d)\le \min \left\{ 2\pi^2d^2n2^{n/2}(d + 1)^n, 2\pi^2 d^2n2^{d/2}(n + 1)^d\right\}.
\]
\end{theorem}
Returning to problem \eqref{prob:origin}, with $f \in \R[x]_d$ being the objective function, one may apply the above theorem by considering the positive
function on $[-1,1]^n$ defined by
\[
x \mapsto f(x) - \funcmin + \varepsilon_f(r)\left(f_\max - f_\min \right),
\]
where $r \in \N$, and $\varepsilon_f(r) > 0$ is a suitable quantity, to be specified, that depends on $f$ and $r$, and satisfies
$\varepsilon_f(r) = O\left(\frac{1}{r} \right)$. This immediately leads to the following corollary.

\begin{corollary}[cf.\  Corollary 15 in Baldi and Slot \cite{doi:10.1137/23M1555430}]
\label{cor:Baldi_Slot}
Consider problem \eqref{prob:origin}, and let $r \in \N$ satisfy $n|r$ as well as
\[
r\ge  4 \Chebyconst \cdot d^2(n\log n) + n  + 2n\sqrt{\Chebyconst C(n,d)}.
\]
Then
\[
f_{(r)} - f_\min \le \varepsilon_f(r)(f_\max - f_\min),
\]
where
\[
\varepsilon_f(r) \le \frac{4 \Chebyconst \cdot d^2(n\log n)}{r}.
\]
Thus, under the stated conditions on $r$, $f_{(r)}$ is an $\varepsilon_f(r)$-approximation of $f_\min$.
\end{corollary}

In this paper, we will show a related result, by using the kernel polynomial method (see, e.g.~\cite{RevModPhys.78.275})
along the lines of \cite{De_Klerk_Vera_2024}.
The kernel polynomial method has been used before to analyze the convergence of the Lasserre hierarchy, e.g.\ for the unit sphere by Fang and Fawzi \cite{10.1007/s10107-020-01537-7}. It has also been used to analyze related SDP hierarchies, e.g.\ in \cite{Laurent2021AnEV,doi:10.1137/22M1540818}. In the conclusion of this paper we give a more detailed discussion of the relationships between these works.

We will show that the rate of convergence may be improved to $O\left(\frac{\log r}{r^2} \right)$ if one only requires fixed relative accuracy: For
any fixed $\varepsilon > 0$ we will show that
{if
$
\frac{r}{\log r} \ge 300\frac{d^{5/2}}{\varepsilon},
$
 one has
\[
f_\min - f_{(208nr)} \le \varepsilon(f_\max - f_\min) + e\cdot n\left( {\frac{7}{2}d^{9/2}} + 14\right)\|f\|_\cheb \log(r)/r^2.
\]
 Thus we show that the Lasserre hierarchy converges at (at least) the  rate of
$O\left(\frac{\log r}{r^2} \right)$ to an $\varepsilon$-approximation of $f_\min$;
 see  \cref{cor:main result}.}

We emphasize that our construction achieves a Putinar-type certificate instead of a Schm\"udgen-type certificate which is more common for kernel-based approaches. For this, it is crucial that we construct a univariate kernel that provides a good SOS-approximation of positive polynomials (as opposed to an approximation in the univariate quadratic module, cf.~\cite{De_Klerk_Vera_2024}). The work of Baldi and Slot~\cite{doi:10.1137/23M1555430} also obtains Putinar-type certificates, but along the way it makes use of Schm\"udgen's Positivstellensatz on the ball. We avoid Schm\"udgen's Positivstellensatz and its associated (large) constants in terms of $n$ and $d$. We state our results with explicit constants that depend polynomially on $n$ and $d$.

We measure the quality of an SOS-approximation in terms of the 1-norm of the coefficients of the difference in the Chebyshev basis, denoted $\|\cdot\|_\cheb$. We complement our construction of good SOS-approximations with a lower bound result, which might be of independent interest. In \cref{thrm:LBcheb} we show that if $q$ is a degree-$r$ sum of squares with $\|(1-x^2)-q\|_\cheb \leq \delta$, then we have $r = \Omega(1/\sqrt{\delta})$.

\subsection*{Outline of this paper}
We start by reviewing some background material in Section \ref{sec:preliminaries}, namely properties of Chebyshev polynomials,
and key results on quadratic modules for later use.
In Section \ref{sec:uni} we describe our use of the polynomial kernel method. Thus we review Gaussian kernels that form the core of our approach in Section \ref{sec:Gaussian}, followed by truncated Gaussians in Section \ref{sec:truncated} to deal with the hypercube.
Then we show how to approximate the Gaussian density by an SOS polynomial density (Section \ref{sec:SOS kernel}). In Section \ref{sec:main result} we then prove our main result, first in the univariate case, and then its multivariate extension.
We consider the limits of our proof technique in in Section \ref{sec:limits}, by considering implications of examples, and
by presenting some numerical experiments where we construct kernels numerically that are optimal in some sense. We conclude with some remarks and discussions of related work in Section \ref{sec:conclusion}.

\section{Preliminaries}
\label{sec:preliminaries}

\subsection{Basic properties of Chebyshev polynomials}
The Chebyshev polynomials of degree $k$ of the first and second kind are respectively defined by:
	\begin{align}
	T_{k}(x) & = \cos (k \arccos x ), \\
	U_k(x) &= \frac{\sin((k+1)\arccos(x))}{\sin(\arccos(x))},
	\end{align}
	for $x \in [-1,1]$.
	
	Defining the inner product for suitable $f,g : [-1,1] \rightarrow \mathbb{R}$,
	\[ \langle f,g \rangle = \int_{-1}^1 \frac{f(x)g(x)}{\pi \sqrt{1-x^2} } \mathrm{d}x \]
		one obtains the orthogonality relations for the Chebyshev polynomials of the first kind:
	\begin{align}
	\langle T_k, T_m \rangle &= \frac{1+\delta_{k,0}}{2} \delta_{k,m}.
		\end{align}

	 One may generalize the Chebyshev polynomials to the multivariate case by taking
products of univariate Chebyshev polynomials. Setting
	\begin{equation}\label{mu cheb}
	  \mathrm{d}\mu(\textbf{x}) := \prod_{i=1}^{n} \frac{1}{\pi \sqrt{1-x_i^2}}\mathrm{d}\textbf{x},
	\end{equation}
then, for $\alpha \in (\mathbb{N}_0)^n$, where $\N_0 = \N \cup \{0\}$, the corresponding  multivariate Chebyshev polynomial of the first kind is defined as
	\[
	T_\alpha(\textbf{x}) = \prod_{i=1}^n  T_{\alpha_i}(x_i),
	\]
that form an orthogonal set of polynomials w.r.t.\ the inner product defined by $\mu$.
	
In what follows we present a number of classical properties of Chebyshev polynomials. For completeness, we provide proofs where we could not find
the exact statement in the literature. The properties we use without proof may readily be found in classic texts, like the one by Rivlin~\cite{rivlin2020chebyshev}.

\begin{lemma} \label{lem:bound outside}
Let $k \in \N$ and $\eps>0$ such that $\eps \leq \frac{1}{10k^2}$. Then we have
\[
\max_{y \in [-1-\eps,1+\eps]} |T_k(y)| = T_k(1+\eps) <  2.
\]
\end{lemma}
\begin{proof}
For any $k \in \mathbb{N}$ we have $|T_k(x)| = |T_k(|x|)|$ for all $x \in \mathbb{R}$. Thus, it is enough to consider only non-negative values for $x$.
For $0 \le x \le 1$ we have $|T_k(x)| \le 1$.
Now we show $T_k(x)$ is increasing for  $x \ge 1$. In this case we have
$T_k(x) = \cosh(k \arccosh (x) )$.   Now, $\arccosh(x) \ge 0$ and increasing for $x \ge 1$ and $\cosh(x)$ is increasing for $x \ge 0$. Thus $T_k$ is increasing for $x \ge 1$.
Summing up, we get $\max_{|x| \le 1+ \eps} |T_k(x)| = T_k(1+\eps) = \cosh(k \arccosh (1 + \eps) )$.
For $0 \le \eps \le \frac{1}{10 k^2} $ we have $\arccosh (1 + \eps)  = \ln (1+ \eps + \sqrt{\eps^2 + 2\eps})  \le 3\sqrt{\eps} \le \frac{3}{\sqrt{10} k}$. Thus   $T_k(1+\eps)  \le \cosh( 1)  < 2$.
\end{proof}

\begin{lemma}
    For any $k \in \N$, one has $|T_k(y)| \leq |2y|^k$ for $|y|\geq 1$.
\end{lemma}

\begin{proof}
    By the Pell equation one has
\[
T_k(y)^2 = 1+ (y^2-1)U^2_{k-1}(y),
\]
where $U_{k-1}$ is the $(k-1)$th Chebyshev polynomial of the second kind. It follows that
\[
|T_k(y)| \ge 1 \mbox{ if } |y| \geq 1.
\]
Since $T_k(1)=1$, this means $T_k(y) \ge 1$ if $y \ge 1$.
By the recursive relation
\[
T_{k+1}(y) = 2yT_k(y) - T_{k-1}(y),
\]
if $y\geq 1$, one therefore  has
\[
    |T_{k+1}(y)| = T_{k+1}(y)\le 2yT_k(y) = |2yT_k(y)|.
\]
If $y \le -1$, the recursive relation together with the symmetry $T_k(-y) = (-1)^kT_k(y)$ imply
\[
(-1)^{k+1}T_{k+1}(|y|) = (-1)^{k+1}2|y|T_k(y) - (-1)^{k-1}T_{k-1}(|y|),
\]
which is the same as
\[
T_{k+1}(|y|) = 2|y|T_k(|y|) - T_{k-1}(|y|),
\]
so that
\[
|T_{k+1}(y)| = T_{k+1}(|y|) = 2|y|T_k(|y|) - T_{k-1}(|y|) \le 2|y|T_k(|y|) = |2yT_k(y)|.
\]
Thus, whenever $|y| \ge 1$, one has $|T_{k+1}(y)| \le |2y||T_k(y)|$.
The proof now follows by induction on $k$, since the statement of the lemma holds trivially for $T_0$.
\end{proof}

We also need the following (well-known) result that relates the sup- and 1-norms for univariate polynomials.

\begin{lemma}
\label{lem:normconversion}
    Let $p$ be a univariate polynomial of degree $d$ with expansion in the Chebyshev basis $p = \sum_{k=0}^d c_k T_k$. Then
    \[
    \|p\|_\infty \leq \|p\|_\cheb \leq \sqrt{2(d+1)} \|p\|_\infty,
    \]
    where
    $\|p\|_\infty = \sup_{t\in [-1,1]} |p(t)|$ and $\|p\|_\cheb = \sum_{k=0}^d |c_k|$.
\end{lemma}
\begin{proof}
Let $p = \sum_{k=0}^d c_k T_k$. Since $|T_k(t)| \leq 1$ for $t \in [-1,1]$, we immediately obtain
$\|p\|_\infty = \sup_{t \in [-1,1]} |p(t)| \leq \sum_{k=0}^d |c_k|$.
We now prove the second inequality. To do so, note that
\begin{equation} \label{eq:1vs2}
\sum_{k=0}^d |c_k| \leq \sqrt{d+1} \sqrt{\sum_{k=0}^d c_k^2}
\end{equation}
due to the Cauchy-Schwarz inequality. The Chebyshev polynomials are known to be orthogonal with respect to the inner product $\langle \cdot, \cdot \rangle$ defined as
\[
\langle p, q\rangle = \frac{1}{\pi} \int_{-1}^1 p(t)q(t) \frac{dt}{\sqrt{1-t^2}}.
\]
The function $\frac{1}{\pi\sqrt{1-t^2}}$ is the probability density function of the Chebyshev measure with respect to the Lebesgue measure on $[-1,1]$, and therefore $\langle p,p \rangle \leq \|p\|_\infty^2$.

Together with the orthogonality of the Chebyshev polynomials with respect to $\langle \cdot, \cdot \rangle$ this yields the inequality
\begin{equation} \label{eq:2vsinfty}
 \sum_{k=0}^d c_k^2 \langle T_k, T_k \rangle = \langle p,p\rangle \leq \|p\|_\infty^2 .
\end{equation}
Combining \cref{eq:1vs2,eq:2vsinfty} yields
\[
\|p\|_\cheb \leq \frac{\sqrt{d+1}}{\min_{k=0,...,d} \sqrt{\langle T_k,T_k \rangle}} \|p\|_\infty.
\]
Observing that $\langle T_0, T_0 \rangle = 1$ and $\langle T_k, T_k\rangle = 1/2$ for $k>0$ concludes the proof.
\end{proof}

The following classical result is due to V.\ Markov (1892).
\begin{lemma} \label{eq:Cheb_der}
    For $k,\ell \in \N$ and $x \in [-1,1]$, and any degree $k$ polynomial $p$ with $\|p\|_\infty = 1$, we have
{
\[
    \left|\frac{d^\ell}{dx^\ell} p(x)\right| \le \left.\frac{d^\ell}{dx^\ell} T_k(x)\right|_{x=1}
    = \frac{k^2(k^2-1^2)\ldots (k^2-(\ell-1)^2)}{1\cdot 3 \cdots (2\ell -1)} \le \frac{k^{2\ell}}{(2\ell -1)!!},
\]
where $(2\ell -1)!! = \prod_{i=1}^\ell (2i-1)$ denotes the double factorial.
In particular, one has $\left\|\frac{d^\ell}{dx^\ell} T_k\right\|_\infty \le \frac{k^{2\ell}}{(2\ell -1)!!}$.}
\end{lemma}

\subsection{Technical lemmas related to quadratic modules}
We start by recalling a classical result about univariate polynomials that are nonnegative on an interval.

\begin{theorem}[Fekete, Markov-Luk\'acz (see, e.g., \cite{3f4f597f-caa0-3104-a850-33b8de850a7c})]\label{THM:MarkovLukacz}
Let $p$ be a univariate polynomial of degree
$2m$. Then $p$ is nonnegative on the interval $[a,b]$ if and
only if
\[
{
 p(t)=q_1^2(t)+(t-a)(b-t)q_2^2(t),}
\]
for some polynomials $q_1$ of degree $m$ and $q_2$ of degree $m-1$.
If the degree of $p$ is $2m+1$ then $p$ is
nonnegative on $[a,b]$ if and only if
\[
 p(t)=(t-a)q_1^2(t)+(b-t)q_2^2(t),
\]
for some polynomials $q_1$ and $q_2$ of degree $m$.
\end{theorem}

Since $|T_k(x)| \le 1$ for all $x \in [-1,1]$, one immediately has the following corollary.
\begin{corollary}
For any $k \in \mathbb{N}_0$, one has $1-T_k(x)^2 \in \mathcal Q(1-x^2)_{2k}$. \label{lem:univar}
\end{corollary}

We first recall a well-known lemma about the monomial basis.

\begin{lemma}[{cf.~\cite[Lemma 23]{10.1145/3597066.3597075}}]
    Let $\alpha \in \N^n$ be such that $|\alpha| \leq 2t$ for an integer $t \in \N$, then
    $
    1-x^\alpha \in \mathcal Q(1-x_1^2,1-x_2^2,\ldots, 1-x_n^2)_{2t}.
    $
\end{lemma}

We now prove an analogous statement in the Chebyshev basis.
\begin{lemma} \label{lem:chebQ}
    Let $\alpha \in \N^n$, then $1-T_\alpha(x), 1+T_\alpha(x) \in \mathcal Q(1-x_1^2,1-x_2^2,\ldots, 1-x_n^2)_{2|\alpha|}$.
\end{lemma}
\begin{proof}
    We first use the identity
    \[
    1 \pm T_\alpha(x) = 1 \pm \prod_{i \in [n]} T_{\alpha_i}(x_i) = \frac{\left(1\pm\prod_{i \in [n]} T_{\alpha_i}(x_i)\right)^2 + \left(1-\prod_{i \in [n]} T_{\alpha_i}(x_i)^2\right)}{2}.
    \]
    (Here the $\pm$ should be interpreted as either consistently $+$ or consistently $-$.) The first term on the right hand side is a square of degree $2|\alpha|$. It remains to show that the second term lies in the quadratic module. To do so, consider the identity
    \[
    s_n \coloneqq 1-\prod_{i \in [n]} T_{\alpha_i}(x_i)^2 = \left(1-T_{\alpha_n}(x_n)^2\right) + T_{\alpha_n}(x_n)^2\left(1-\prod_{i \in [n-1]} T_{\alpha_i}(x_i)^2\right).
    \]
        The first term on the right hand side is a univariate polynomial of the form $1-T_k(y)^2$ for some $k \in \N$; by \cref{lem:univar} it belongs to $\mathcal Q(1-y^2)_{2k}$. The second term has degree $2|\alpha|$: it is a square of degree $2\alpha_n$ times $s_{n-1} \coloneqq \left(1-\prod_{i \in [n-1]} T_{\alpha_i}(x_i)^2\right)$, where $s_{n-1}$ is defined by taking the product only over the first $n-1$ variables. Repeating the argument for $s_{n-1}$ allows us to iteratively reduce the number of variables until we are in the univariate case. At that point, we can again apply \cref{lem:univar}. 
\end{proof}

Note that the statement we prove is slightly weaker compared to the one in the monomial basis: the degree bound becomes $2|\alpha|$ instead of the least even integer greater than $|\alpha|$.

\begin{corollary} \label{cor:1cheb}
    Let $p \in \R[x]$, then $\|p\|_\cheb - p \in \mathcal Q(1-x_1^2,1-x_2^2,\ldots, 1-x_n^2)_{2\mathrm{deg}(p)}$.
\end{corollary}
\begin{proof}
    Write $p = \sum_{\alpha \in \N^n} p_\alpha T_\alpha$, then
    $\|p\|_\cheb - p = \sum_{\alpha \in \N^n} |p_\alpha| (1- \mathrm{sign}(p_\alpha)T_\alpha)$ and observe that, by \cref{lem:chebQ}, each term $(1- \mathrm{sign}(p_\alpha)T_\alpha)$ belongs to $\mathcal Q(1-x_1^2,1-x_2^2,\ldots, 1-x_n^2)_{2|\alpha|}$.
\end{proof}

The main result of this section is a simple consequence of the above.
\begin{theorem} \label{lem:1approx}
    Let $f$ and $q$ be polynomials where $q \in \mathcal Q(1-x_1^2,1-x_2^2,\ldots, 1-x_n^2)_{r}$ and $r\ge \mathrm{deg}(f)$, then
    \[
    f+ \|f-q\|_\cheb \in \mathcal Q(1-x_1^2,1-x_2^2,\ldots, 1-x_n^2)_{2r}.
    \]
\end{theorem}
\begin{proof}
    Apply \cref{cor:1cheb} to $p = q-f$ to obtain
    $f - q + \|f-q\|_\cheb \in \mathcal Q(1-x_1^2,1-x_2^2,\ldots, 1-x_n^2)_{2r}$ and add $q$ on both sides.
\end{proof}

This result in turn has the following implication for the Lasserre hierarchy \eqref{eq:lasserre hierarchy}.
\begin{corollary}
\label{cor:1norm}
Consider problem \eqref{prob:origin} and its associated Lasserre hierarchy \eqref{eq:lasserre hierarchy}.
One has, for $r \ge \mathrm{deg}(f)$,
\[
f_\min - f_{(2r)} \le \min_{q \in Q(1-x_1^2,1-x_2^2,\ldots, 1-x_n^2)_{2r}} \left\| f - f_\min - q\right\|_\cheb.
\]
\end{corollary}
\begin{proof}
Follows immediately from Theorem \ref{lem:1approx}, by replacing $f$ in Theorem \ref{lem:1approx} by $f-f_\min$.
\end{proof}
Our analysis of the convergence rate of the Lasserre hierarchy will be based on Corollary \ref{cor:1norm}, in the sense that
we will construct a suitable $q \in Q(1-x_1^2,1-x_2^2,\ldots, 1-x_n^2)_{2r}$ to approximate $f - f_\min$, via the polynomial kernel method.
Our construction in the univariate case will in fact be SOS, which allows us to easily generalize to the multivariate setting. 

\section{An approximation kernel approach}
\label{sec:uni}

The outline of this section is as follows.
We show, in three steps, how to construct a sequence of convolution operators $\mathcal{K}_r$  with corresponding kernels
$K_r(x,y)$ for $r \in \N$, that satisfy the following two properties:
\begin{enumerate}
    \item
if $f \in \R[x]_d$ is positive on $[-1,1]$,
    \[
 \mathcal{K}_rf   \in \mathcal Q(1-x^2)_{334r} 
 \quad(r\ge \mathrm{deg}(f)),
\]
\item
and
\[
\left\| \mathcal{K}_rf - f\right\|_\cheb = O(\log(r)/r^2) \quad \forall f \in \R[x]_d,
\]
for some fixed $d$.
\end{enumerate}
Through \cref{lem:1approx}, such a sequence of kernels would immediately imply that, for $f \in \R[x]_d$ positive on $[-1,1]$, we have
\[
f+ \left\| \mathcal{K}_rf - f\right\|_\cheb \in \mathcal Q(1-x_1^2,1-x_2^2,\ldots, 1-x_n^2)_{334r}.
\]

We construct our kernels $K_r$ in three steps.

First, we show that the convolution operator $\mathcal K_G^\sigma$ corresponding to the normal distribution $\mathcal N(0,\sigma^2)$, has the property that $\mathcal K_G^\sigma$ nearly preserves degree-$k$ polynomials, under suitable assumptions on $k$ and $\sigma$: we show
\[
\|\mathcal K_G^\sigma T_k - T_k\|_\cheb \leq {k^{9/2}\sigma^2},
\]
see \cref{lem:approximate identity} for a formal statement.

Second, we write $\mathcal K_G^{\sigma,R}$ for the convolution operator obtained by truncating the domain of integration used by $\mathcal K_G^\sigma$ to an interval $[-R,R]$. We show that if $R-1$ is large enough compared to $\sigma^2$ and $k$, then this leads to negligible error: in \cref{lem:truncate} we show that if $R-1 \geq \gamma (2+\sqrt{2}) \sqrt{k} \sigma$, then
\[
\|\mathcal K_G^{\sigma} T_k - \mathcal K_G^{\sigma,R} T_k \|_\infty \leq 2\sqrt{2} \exp(-\gamma^2).
\]
It is important to note that here we measure the error in the sup-norm since $\mathcal K_G^{\sigma,R}$ does not map polynomials to polynomials.

Third, we approximate the kernel $K_G^{\sigma,R}$ by an SOS kernel $K_r$, by using a degree-$O(r)$ SOS approximation~$s^2$ of $\exp(-t)$ on the interval $[0,b]$ where $b = (R+1)^2/(2\sigma^2)$. We show in \cref{lemma:truncation_error} that
\[
\|\mathcal K_G^{\sigma,R}T_k - \mathcal K_r T_k \|_\infty \leq
\frac{24\delta}{\sqrt{2\pi}\sigma},
\]
where $\delta$ is such that $|\exp(-t)- s(t)^2| \leq \delta$ for all $t \in [0,b]$.

Finally, in \cref{sec:proof uni}, we combine these three estimates to show that
\begin{align*}
\|\mathcal K_r T_k - T_k \|_\cheb &\leq &&\|\mathcal K_G^\sigma T_k - T_k\|_\cheb + &&\sqrt{2(r+1)} \left(\vphantom{\|\mathcal K_G^\sigma T_k - \mathcal K_G^{\sigma,R} T_k \|_\infty}\right. \|\mathcal K_G^\sigma T_k - \mathcal K_G^{\sigma,R} T_k \|_\infty + &&\|\mathcal K_G^{\sigma, R} T_k - \mathcal K_r T_k\|_\infty \left.\vphantom{\|\mathcal K_G^\sigma T_k - \mathcal K_G^{\sigma,R} T_k \|_\infty} \right) \\
&\leq &&{k^{9/2}\sigma^2} + &&\sqrt{2(r+1)} \left(\vphantom{\|\mathcal K_G^\sigma T_k - \mathcal K_G^{\sigma,R} T_k \|_\infty}\right. 2\sqrt{2}\exp(-\gamma^2) \ \  + &&\frac{24\delta}{\sqrt{2\pi}\sigma} \left.\vphantom{\|\mathcal K_G^\sigma T_k - \mathcal K_G^{\sigma,R} T_k \|_\infty}\right).
\end{align*}
We arrive at an overall upper bound of $O(\log(r)/r^2)$ by balancing the three
terms under the constraint that $r = O(\sqrt{\log(1/\delta)/\sigma^2})$. Specifically, we choose
\begin{equation*} 
\sigma^2= O(\log(r)/r^2), \qquad R = 1 + O(\sqrt{d\log(r)}\sigma), \qquad {\delta= r^{-7/2}}.
\end{equation*}
We refer to \cref{thrm:main uni} for a precise statement, and to Corollary \ref{cor:main result} for its implication for the Lasserre hierarchy.


\subsection{Kernels of the Gaussian type}
\label{sec:Gaussian}

Here we consider kernels of the Gaussian type. That is, for a parameter $\sigma \geq 0$, we consider the kernel $K_G^\sigma$ defined via
\[
K^\sigma_G(x,y) = \frac{1}{\sqrt{2\pi}\sigma}\exp{\left(-\frac{(x-y)^2}{2\sigma^2}\right)},
\]
for $x,y \in \R$. The associated linear operator on polynomials is known as the Gauss-Weierstrass convolution operator. It acts on $f \in \R[x]$ as
\[
\mathcal{K}^\sigma_G(f)(x) := \int_{\R} K^\sigma_G(x,y) f(y)dy.
\]
The operator $\mathcal K_G^\sigma$ is the convolution operator corresponding to the probability density function $p$ of the normal distribution $\mathcal N(0,\sigma^2)$ with mean $0$ and variance $\sigma^2$. Our aim in this section is to show that $\mathcal K_G^\sigma T_k$ is a good approximation of $T_k$, under suitable assumptions on $\sigma$ and $k$.
To do so, we first show that $\mathcal K_G^\sigma$ maps degree-$k$ polynomials to degree-$k$ polynomials. We then show how to bound $\|\mathcal K_G^\sigma - T_k \|_\infty$. We finally convert this to a bound on $\|\mathcal K_G^\sigma - T_k \|_\cheb$.

{
To facilitate the analysis, we briefly recall the moments of $\mathcal N(0,1)$:
\begin{equation} \label{eq:Gaussmoments}
 \mathbb E_{Z \sim \mathcal N(0,1)} [Z^{\ell}] =
 \left\{
\begin{array}{lr}
    (\ell-1)!! & \mbox{ if $\ell$ is even} \\
    0 & \mbox{ if $\ell$ is odd}, \\
\end{array} \right.
\end{equation}}
where $(\ell-1)!!$ denotes the double factorial of $(\ell-1)$, that is, the product of all odd numbers between $1$ and $\ell-1$; see e.g.\ \cite[p. 89]{Kotz}.

The following result {relies on} the well-known fact that the moments of the normal distribution are polynomials.
{We give a short proof for completeness.}

\begin{lemma} \label{lem:Gaussianpoly}
Let $\sigma > 0$. Then $\mathcal K_G^\sigma$ maps degree-$k$ polynomials to degree-$k$ polynomials.
\end{lemma}
{\begin{proof}
Since  $\mathcal K_G^\sigma$ is a linear operator, we only need to show that it maps the monomial $f(y) = y^k$ to a polynomial of degree $k$.
To this end, let ${Z \sim \mathcal N(0,1)}$ so that $\sigma Z + x \sim \mathcal N(x,\sigma^2)$ for fixed $x \in \mathbb{R}$.
Now
\begin{eqnarray*}
\mathcal K_G^\sigma(f)(x) &=& \int_{\mathbb{R}}y^k\frac{1}{\sqrt{2\pi}\sigma}\exp{\left(-\frac{(x-y)^2}{2\sigma^2}\right)}dy \\
&\equiv & \mathbb{E}\left[(\sigma Z + x)^k \right] \\
&=&  \sum_{\ell =0}^k {k \choose \ell} \sigma^\ell \mathbb{E}\left[Z^\ell\right]x^{k-\ell}, \\
&=& \sum_{\ell =0}^{\lfloor k/2 \rfloor} {k \choose 2\ell} \sigma^{2\ell}(2\ell-1)!!  x^{k -  2\ell},
\end{eqnarray*}
where the last equality uses \eqref{eq:Gaussmoments}. Note that the final expression is indeed a polynomial of degree $k$ in $x$.
\end{proof}
}

\begin{lemma}
\label{lem:Gaussian approximation Cheb basis}
    For $\sigma \geq 0$, $k \in \N$,  we have
    {
     \begin{equation*}
        \|T_k - \mathcal K_G^\sigma T_k \|_\infty \leq \sum_{\ell=1}^{\lfloor k/2\rfloor} \frac{(2\ell -1)!!}{(2\ell)!(4\ell -1)!!}(k^2 \sigma)^{2\ell}.
     \end{equation*}
     }
\end{lemma}
\begin{proof}
    Fix an $x \in [-1,1]$. We show how to bound $|T_k(x) - \int K_G^\sigma(x,y) T_k(y) dy|$. To do so, note that {$K_G^\sigma(x,\cdot)$} corresponds to the probability density function of the normal distribution with mean $x$ and variance~$\sigma^2$. Writing $p$ for the probability density function of $\mathcal N(0,\sigma^2)$, we have {$K_G^\sigma(x,y) = p(x-y)$}. Using a change of variables $y=x+z$
    and the Taylor expansion
{
\[
T_k(x+z) = T_k(x) + \sum_{\ell=1}^{k} \frac{z^\ell}{\ell !}\frac{d^{\ell}}{dx^{\ell}} T_k(x),
\]
we have
    \begin{align*}
    T_k(x) - \int_{\R} K_G^\sigma(x,y) T_k(y) dy &= \int_{\R} p(z) \left(T_k(x)- T_k(x+z)\right) dz \\
    &= -\int_{\mathbb{R}}  p(z) \left( \sum_{\ell=1}^{k} \frac{z^\ell}{\ell !} \frac{d^{\ell}}{dx^{\ell}} T_k(x) \right) dz \\
    &= -\sum_{\ell=1}^{k} \frac{1}{\ell !}\frac{d^{\ell}}{dx^{\ell}} T_k(x) \cdot \mathbb E_{Z \sim \mathcal N(0,\sigma^2)}[Z^{\ell}],
\end{align*}
where the last equality uses $\int_{\mathbb{R}}  z^\ell p(z)dz \equiv \mathbb E_{Z \sim \mathcal N(0,\sigma^2)}[Z^{\ell}]$.
}
 Using the triangle inequality, {and the fact that the moments of odd order are zero from \eqref{eq:Gaussmoments}}, we obtain
\[
\left|T_k(x) - \int_{\R} K_G^\sigma(x,y) T_k(y) dy\right| \leq \sum_{\ell=1}^{{\lfloor k/2\rfloor}} \mathbb E_{Z \sim \mathcal N(0,\sigma^2)}[Z^{2\ell}]  \frac{1}{(2\ell) !}\left| \frac{d^{2\ell}}{dx^{2\ell}} T_k(x) \right|.
\]
To conclude the proof, we use \eqref{eq:Gaussmoments} to obtain {${\mathbb E}_{Z \sim {\mathcal N}(0,\sigma^2)}[Z^{2\ell}] = \sigma^{2\ell}(2\ell-1)!!$}, and use \cref{eq:Cheb_der} to
get
\[
 \left|\frac{d^{2\ell}}{dx^{2\ell}} T_k(x) \right| \le  \frac{k^{4\ell}}{(4\ell -1)!!}.
 \]
\end{proof}

In order to have a simpler bound for complexity analysis, we  derive the following result.


{
\begin{lemma} \label{lem:approximate identity}
       Let $\sigma \geq 0$ and $k \in \N$ be such that $k^2 \sigma \leq 1$. Then $\mathcal{K}^\sigma_G T_k$ is a degree-$k$ polynomial that satisfies
    \begin{equation*}
        \|T_k - \mathcal  K_G^\sigma T_k \|_\infty \leq \frac{1}{2}k^4 \sigma^2 \mbox{ and } \|\mathcal{K}^\sigma_G T_k - T_k\|_\cheb \leq  k^{9/2} \sigma^2.
     \end{equation*}
\end{lemma}
}
\begin{proof}[Proof]
\cref{lem:Gaussianpoly} shows that {$\mathcal K_G^\sigma T_k$} is a degree-$k$ polynomial.
{With reference to \cref{lem:Gaussian approximation Cheb basis}, we have
\[
\sum_{\ell=1}^{\lfloor k/2\rfloor} \frac{(2\ell -1)!!}{(2\ell)!(4\ell -1)!!}(k^2 \sigma)^{2\ell}
\le \frac{1}{3}\sum_{\ell=1}^{\lfloor k/2\rfloor} \frac{1}{\ell!2^\ell}(k^4 \sigma^2)^{\ell}
\le \frac{1}{3}(\exp( k^4 \sigma^2/2)-1 )  < k^4\sigma^2/2,
\]
 where we used $\frac{(2\ell -1)!!}{(4\ell -1)!!} \le \frac{1}{3}$ and $(2\ell)! \ge 2^\ell \ell!$ in the first inequality, and   the assumption $k^2 \sigma \leq 1$ in the last inequality.
  We finally use \cref{lem:normconversion} applied to the degree-$k$ polynomial $K_G^\sigma T_k -T_k$ to convert
  the sup-norm bound to a bound on the coefficient norm:
   $$\|\mathcal K_G^\sigma T_k -T_k\|_\cheb \leq \sqrt{2(k+1)} \|\mathcal K_G^\sigma T_k-T_k\|_\infty \leq \sqrt{2(k+1)} \frac{1}{2}k^4 \sigma^2 \le k^{9/2}\sigma^2,$$
   where the last inequality follows from $\sqrt{k+1} \le \sqrt{2}\sqrt{k}$.
}
\end{proof}

\subsection{Truncated Gaussian kernel}
\label{sec:truncated}
The next step towards an SOS kernel is to replace the Gaussian by a truncated Gaussian on $[-R,R] \supset [-1,1]$ for some sufficiently large $R$. The moments of this truncated kernel are no longer polynomials. However, we will show that if we apply the truncated kernel to a Chebyshev polynomial, then we obtain a good approximation of that Chebyshev polynomial in the sup-norm.

\begin{definition}
    For $\sigma \geq 0$ and $R \geq 1$ we define the \textit{truncated Gauss-Weierstrass} operator $\mathcal K_G^{\sigma,R}$ as
\[
 \mathcal{K}^{\sigma,R}_G(f)(x) := \int_{[-R,R]} K^\sigma_G(x,y)f(y)dy.
\]
\end{definition}

Before we analyze this operator, we first recall a useful concentration inequality, {namely the Chernoff bound \cite{Chernoff} for the normal distribution:}
    let $\mu, \sigma \in \R$ and $X \sim \mathcal N(\mu,\sigma^2)$, then, for all $c \geq 0$, we have
    \begin{equation} \label{eq:concentration}
    \mathbb P[|X-\mu| \geq c ] \leq 2 e^{-c^2/(2\sigma^2)}.
    \end{equation}

\begin{lemma} \label{lem:truncate}
Let $\sigma \geq 0$ and $k \in \N$ be such that $k \sigma^2 \leq 1$. Assume $R \geq 1$ and $\gamma \geq 1$ are such that  $R -1 \geq \gamma (2+\sqrt{2})\sqrt{k}\sigma$. Then we have
\[
\|\mathcal K_G^{\sigma} T_k - \mathcal K_G^{\sigma,R} T_k \|_\infty \leq 2 \sqrt{2} \exp(-\gamma^2).
\]
\end{lemma}

\begin{proof} 

Let $x \in [-1,1]$, then we have 
\begin{align*}
\left|\int_{\R} K^\sigma_G(x,y)T_k(y)dy-\int_{[-R,R]} K^\sigma_G(x,y)T_k(y)dy\right| &\leq \int_{\R \setminus [-R,R]} K^\sigma_G(x,y)|T_k(y)| dy \\
&= \frac{1}{\sqrt{2\pi}\sigma} \int_{\R \setminus [-R,R]} e^{-(x-y)^2/2\sigma^2} |T_k(y)| dy
\end{align*}
Since $y \in \R \setminus [-R,R]$ and $R \geq 1$, we have $|T_k(y)| \leq |2y|^k$. Let us write $y = x+\delta$ and note that $|2y|^k \leq (2(1+|\delta|))^k \leq e^{k(1+|\delta|)}$. Then we have
\begin{align*}
    \frac{1}{\sqrt{2\pi}\sigma} \int_{\R \setminus [-R,R]} e^{-(x-y)^2/2\sigma^2} |T_k(y)| dy  &\leq \frac{1}{\sqrt{2\pi}\sigma} \int_{|\delta| \geq R-1} e^{-\frac{\delta^2}{2\sigma^2} +k(1+|\delta|)}  d\delta \\
    &\leq \frac{1}{\sqrt{2\pi}\sigma} \int_{|\delta| \geq R-1} e^{-\frac{\delta^2}{4\sigma^2}}  d\delta \\
    &\leq \sqrt{2}\, \mathbb{P}_{\delta \sim \mathcal N(0,2\sigma^2)}[|\delta| \geq R-1] \\
    &\leq 2\sqrt{2} e^{-(R-1)^2/(4\sigma^2)}
\end{align*}
where in the last inequality we use \cref{eq:concentration}. The second inequality holds since $R-1 \geq  2(k\sigma^2 + \sqrt{k}\sigma\sqrt{k\sigma^2 + 1})$, which in turn is a consequence of the assumptions that $R-1 \geq \gamma(2+\sqrt{2})\sqrt{k}\sigma$ and $k \sigma^2 \leq 1$. Indeed, under those assumptions we have $R-1 \geq  2(\sqrt{k}\sigma+\sqrt{k}\sigma \sqrt{2}) \geq 2(k\sigma^2 + \sqrt{k}\sigma\sqrt{k\sigma^2 + 1})$. The lemma then follows by observing that $e^{-(R-1)^2/(4\sigma^2)} \leq e^{-\gamma^2}$.
\end{proof}

\subsection{An SOS kernel}
\label{sec:SOS kernel}

The next step is to approximate the truncated Gaussian kernel by an SOS kernel.
To do so, we first approximate the function $t \mapsto \exp(-t)$ by an SOS polynomial using the following lemma.

\begin{lemma}[{cf.~\cite[Thrm~4.1]{sachdeva2014approx}}]
\label{lemma:sos approx exp}
    For every $b>0$ and $0 \le \delta <1$, there exists a polynomial $s_{b,\delta}$ with degree
    \begin{equation}
    \label{eq:degree}
    \mathrm{deg}(s_{b,\delta}) = \left\lceil\sqrt{2{\theta}\log (4/\delta)} \right\rceil, \mbox{ where }
  {\theta} =
   \left\lceil\max\left\{\frac{1}{2}be^2,\log(2/\delta)\right\}\right\rceil,
    \end{equation}
    such that
    \[
    |\exp(-t)- s_{b,\delta}(t)| \le \delta \quad \forall t \in [0,b].
    \]
\end{lemma}
An immediate corollary to the lemma is obtained by approximating $t \mapsto \exp(-2t)$ by $s_{b,\delta}^2$.

\begin{corollary} \label{cor:supexp}
      For every $b>0$ and $0 \le \delta <1$, there exists a polynomial $s_{b,\delta}$ with degree as in \eqref{eq:degree},
    such that
$|\exp(-2t)- s_{b,\delta}^2(t)| \le 2\delta + \delta^2 \leq 3\delta$ for all $t \in [0,b]$.
\end{corollary}
{
\begin{proof}
With $s_{b,\delta}$  as in Lemma \ref{lemma:sos approx exp} and $t \in [0,b]$, one has
\begin{eqnarray*}
  |\exp(-2t)- s_{b,\delta}^2(t)|  &=& |(\exp(-t)- s_{b,\delta}(t))(\exp(-t)+ s_{b,\delta}(t))|  \\
   &=& |\exp(-t)- s_{b,\delta}(t)|\cdot |2\exp(-t)+ s_{b,\delta}(t)-\exp(-t)| \\
   &\le & \delta(2 + \delta),
\end{eqnarray*}
where the inequality follows from \cref{lemma:sos approx exp} as well as $\exp(-t) \le 1$.
\end{proof}
}

We now define the SOS kernel that we will use, namely:
\[
K_r(x,y) = \frac{1}{\sqrt{2\pi}\sigma} s^2_{b,\delta}\left(\frac{(x-y)^2}{4\sigma^2} \right),
\]
for the following choices of parameters:
\begin{equation} \label{eq:values2}
\delta = r^{-7/2}, \quad \sigma = \frac{\sqrt{\log(1/\delta)}}{r}, \quad \gamma = \sqrt{\frac{5}{2}\log(r)}, \quad R -1 = \gamma (2+\sqrt{2})\sqrt{d}\sigma, \quad b =  \frac{(R+1)^2}{4 \sigma^2}.
\end{equation}
{
For this parameter choice, we have, with reference to \eqref{eq:degree},
\[
\max\left\{\frac{1}{2}be^2,\log(2/\delta)\right\} = \frac{1}{2}be^2,
\]
since, using $R+1 > 2$,
\begin{eqnarray*}
\frac{1}{2} b \cdot e^2
&=& \frac{e^2}{8} (R+1)^2/\sigma^2 \\
&>& \frac{e^2}{2} \cdot \frac{r^2}{\log(r^{7/2})}\\
&=& \frac{e^2}{7} \cdot \frac{r^2}{\log(r)}\\
&>& \frac{r^2}{\log(r)},
\end{eqnarray*}
and $\log(2/\delta) = \log(2) + \frac{7}{2}\log r < \frac{r^2}{\log(r)}$ if $r \ge 2$.}

Thus, by \cref{cor:supexp}, we have
\begin{eqnarray*}
    \deg(K_r) &=& 4\deg(s_{b,\delta}) \\
    &=&  4\left\lceil\sqrt{2\left\lceil \frac{1}{2}be^2\right\rceil\log (4/\delta)} \right\rceil \\
    &\le & 4\sqrt{(be^2+2)[\log(1/\delta)+\log 4]} \\
    &\le& 4\left(e\frac{R+1}{ {2}\sigma}+\sqrt{2}\right)(\sigma r + \sqrt{\log 4}) \\
    &=& {2}e\cdot r(R+1) +4\sqrt{2}\sigma \cdot r + {2}\sqrt{\log 4}\cdot e\cdot \frac{R+1}{\sigma} + {4}\sqrt{2\log 4},
\end{eqnarray*}
{where the first inequality uses $2\left\lceil \frac{1}{2}be^2\right\rceil \le be^2+2$, and the second inequality uses the sub-additivity of the square root.}
Using that
\[
R+1 = 2+ \sqrt{35}(1 + 1/\sqrt{2})\sqrt{d}\frac{\log r}{r} \le 10 \mbox{ if $r \ge d$},
\]
as well as $\sigma = \frac{\sqrt{\frac{7}{2}\log r}}{r}$,
one obtains, for $r \ge d {\ge 2}$,
\begin{equation}
\label{eq:Kr_degree_bound}
    \deg(K_r) \le {20}e\cdot r +4\sqrt{7\log r} + {20\sqrt{\log 4}}\cdot e\frac{r}{\sqrt{\frac{7}{2}\log r}} + {4}\sqrt{2\log 4}
    <  {104r} = O(r).
\end{equation}

The  convolution operator associated with $K_r$ is
\[
 \mathcal{K}_r (f)(x) := \int_{[-R,R]} K_r(x,y)f(y)dy.
\]

By \cref{cor:supexp}, we immediately obtain the following results.
\begin{lemma} \label{lem:pointwise Kr vs exp}
For the parameters $\delta,\sigma, R, b$ as in \cref{eq:values2} one has
    \[
\left| K_r(x,y) - \frac{1}{\sqrt{2\pi}\sigma}\exp{\left(-(x-y)^2/2\sigma^2\right)} \right| \leq \frac{3\delta}{\sqrt{2\pi}\sigma}
\qquad \forall x \in [-1,1], \; \forall y \in [-R,R].
\]
\end{lemma}
\begin{proof}
Fix $x \in [-1,1]$ and $y \in [-R,R]$ and set $t = \frac{(x-y)^2}{4 \sigma^2}$. Note that $0 \leq t = \frac{(x-y)^2}{4 \sigma^2} \leq \frac{(R+1)^2}{4 \sigma^2} = b$. We therefore have
\[
\left|K_r(x,y) - \frac{1}{\sqrt{2\pi}\sigma}\exp{\left(-(x-y)^2/2\sigma^2\right)}\right| = \frac{1}{\sqrt{2\pi}\sigma} |\exp(-2t)- s_{b,\delta}(t)^2| \leq \frac{3\delta}{\sqrt{2\pi}\sigma},
\]
where the last inequality follows from \cref{cor:supexp}.
\end{proof}

Next we show that our SOS kernel approximates the truncated Gaussian kernel well enough, under suitable assumptions on $R$.

\begin{lemma}
\label{lemma:truncation_error}

Let the parameters $\delta,\sigma, R, b$ be as in \cref{eq:values2}. For any fixed $k \in \N$ we have
\[
\|\mathcal K_r T_k - \mathcal K_G^{\sigma,R} T_k \|_\infty \leq 2R\frac{3\delta}{\sqrt{2\pi}\sigma} \cdot \max_{y \in [-R,R]} |T_k(y)|.
\]
In particular, if $R \leq 1 + \const{k}$, then we have $\|\mathcal K_r T_k - \mathcal K_G^{\sigma,R} T_k \|_\infty \leq \frac{24\delta}{\sqrt{2\pi}\sigma}$.
\end{lemma}
\begin{proof}

Fix $x \in [-1,1]$. We use the triangle inequality and \cref{lem:pointwise Kr vs exp} to obtain
\begin{align*}
|\mathcal K_r (T_k)(x) - \mathcal K_G^{\sigma,R}(T_k)(x) | &= \left|\int_{[-R,R]} [K_r(x,y) -K_G^\sigma(x,y)]T_k(y)dy\right| \\
&\leq \int_{[-R,R]} \left|K_r(x,y) -K_G^\sigma(x,y)\right|
\left|T_k(y)\right| dy \\
&\leq 2R\frac{3\delta}{\sqrt{2\pi}\sigma} \cdot \max_{y \in [-R,R]} |T_k(y)|.
\end{align*}
Finally we note that if $R \leq 1 + \const{k}$, then $R \leq 2$ and $\max_{y \in [-R,R]} |T_k(y)| \leq 2$, from which the last claim follows.
\end{proof}

We finally observe that this operator maps polynomials that are nonnegative on $[-R,R]$ to SOS polynomials of degree $O(r)$.

\begin{proposition} \label{prop:Kr is SOS}
Let $R \geq 1, \sigma^2 \geq 0, r \in \N$, and assume $f \in \R[x]_d$ is nonnegative on the interval $[-R,R]$. Then $\mathcal{K}_r f$ is an SOS polynomial of the same degree as $K_r$.
\end{proposition}
\begin{proof}
    Assume we have a quadrature rule of degree {$d + \deg(K_r)$}  for the Lebesgue measure on $[-R,R]$ that is given by nodes $\omega_1,\ldots,\omega_N \in [-R,R]$ with corresponding weights $c_1>0, \ldots, c_N >0$, for some $N \in \N$.
   {In other words, we assume that for all $p \in \R[x]_{d + \deg(K_r)}$ one has
\[
\int_{-R}^R p(x)dx = \sum_{j=1}^N c_jp(\omega_j).
\]}
(Such a rule exists for any compact subset of $\R^n$ \cite{Tchakaloff}.)
    We then have
    \begin{eqnarray*}
        \mathcal{K}_r (f)(x) &:=& \int_{[-R,R]} K_r(x,y)f(y)dy \\
        &=& \sum_{j=1}^N c_jK_r(x,\omega_j)f(\omega_j).
    \end{eqnarray*}
    By construction $K_r(\cdot,\omega_j)$ is an SOS polynomial of degree $\deg(K_r)$ for each $j$. Moreover, by assumption, $f(\omega_j) \geq 0$ for each $j \in [N]$. The result follows.
\end{proof}

\section{Proof of main result }
\label{sec:main result}
In this section we prove our main result, first for the univariate case, and then, in  Section \ref{sec:multivariate proof}, we
show how to leverage the univariate approach to establish our main result in the multivariate setting.

\subsection{Proof in the univariate case}
\label{sec:proof uni}

\begin{theorem}[Approximate identity] \label{thrm:approxIDuni}
Fix $d \in \N$ and let $k \in \N$ be such that $k \leq d$. Let $r \in \N$ and choose $\delta = r^{-7/2}$,
 $\sigma = \sqrt{\log(1/\delta)}/r$, and $R=1 + \sqrt{\frac52\log(r)}(2+\sqrt{2})\sqrt{d}\sigma$, as in \cref{eq:values2}.
  Assume $r$ is sufficiently large so that {$k^2 \sigma \leq 1$} and $R \leq 1+ \const{d}$, i.e.\ assume
\begin{equation}
\label{eq:rbounds}
\frac{r}{\log r} \ge {10\sqrt{35}\left(1+ 1/\sqrt{2}\right)d^{5/2}.}
\end{equation}
Then we have $\deg(K_r) = O(r)$ and
\begin{align*}
\|\mathcal K_r T_k - T_k \|_\cheb &\leq {\frac{7}{2}k^{9/2}} \frac{\log(r)}{r^2} + \sqrt{2(r+1)} \left(2 \sqrt{2} r^{-5/2} + \frac{24}{\sqrt{2\pi}} \frac{r^{-5/2}}{\sqrt{\frac{7}{2}\log(r)}} \right) \\
& \le
\left( {\frac{7}{2}d^{9/2}} + 14\right)\frac{\log(r)}{r^2}.
\end{align*}
\end{theorem}
\begin{proof}
We follow the proof strategy outlined at the start of \cref{sec:uni}.
We first use the triangle inequality to obtain
\[
\|\mathcal K_r T_k - T_k \|_\cheb \leq \|\mathcal K_G^\sigma T_k - T_k \|_\cheb + \|\mathcal K_G^\sigma T_k - \mathcal K_r T_k\|_\cheb,
\]
where we use that $\mathcal K_G^\sigma T_k$ is a polynomial due to \cref{lem:Gaussianpoly}. \cref{lem:approximate identity} shows
 that $$\|\mathcal K_G^\sigma T_k - T_k \|_\cheb \leq {k^{9/2} \sigma^2 = \frac{7}{2}k^{9/2} \frac{\log(r)}{r^2}},$$
 since, by assumption, {$k^2 \sigma \leq 1$}.
 {Indeed, the condition $k^2 \sigma \leq 1$ is the same as $\frac{r}{\sqrt{\log r}} \ge \sqrt{\frac{7}{2}}k^2$, which in turn is implied by
 \eqref{eq:rbounds}.}

It remains to bound $ \|\mathcal K_G^\sigma T_k - \mathcal K_r T_k\|_\cheb$. To do so, we first use \cref{lem:normconversion} to upper bound the $1$-norm by the sup-norm:
\[
\|\mathcal K_G^\sigma T_k - \mathcal K_r T_k\|_\cheb \leq \sqrt{2(r+1)} \|\mathcal K_G^\sigma T_k - \mathcal K_r T_k\|_\infty.
\]
We now use the triangle inequality to obtain
\[
\|\mathcal K_G^\sigma T_k - \mathcal K_r T_k\|_\infty \leq \|\mathcal K_G^\sigma T_k - \mathcal K_G^{\sigma,R} T_k\|_\infty + \|\mathcal K_G^{\sigma,R} T_k - \mathcal K_r T_k\|_\infty.
\]
Now, by \cref{lem:truncate}, we can bound the first quantity by
\[
\|\mathcal K_G^\sigma T_k - \mathcal K_G^{\sigma,R} T_k\|_\infty \leq 2 \sqrt{2} r^{-5/2},
\]
where we use that $k \sigma^2 \leq 1$ and $R-1 \geq \gamma (2+\sqrt{2}) \sqrt{k}\sigma$ for $\gamma = \sqrt{\frac{5}{2} \log(r)}$. Similarly, using \cref{lemma:truncation_error}, we have
\[
\|\mathcal K_G^{\sigma,R} T_k - \mathcal K_r T_k\|_\infty \leq \frac{24\delta}{\sqrt{2\pi}\sigma} = \frac{24}{\sqrt{2\pi}} \frac{r^{-5/2}}{\sqrt{\frac{7}{2}\log(r)}},
\]
where we use that $R \leq 1+ \const{d}$.

Combining the above bounds yields the claimed upper bound on $\|\mathcal K_r T_k - T_k \|_\cheb$.
\end{proof}

We finally arrive at our main theorem in the univariate setting, namely a weaker version of Theorem \ref{THM:MarkovLukacz} that may --- unlike
Theorem \ref{THM:MarkovLukacz} ---
readily be extended to the multivariate case.
For this we need a final ingredient:  a lower bound on $\min_{x \in [-R,R]} f(x)$ for degree-$d$ polynomials that are nonnegative on $[-1,1]$.
 Such a bound can be found for example in the work of Baldi and Slot \cite{doi:10.1137/23M1555430}, even in the multivariate setting.

\begin{lemma}[{\cite{doi:10.1137/23M1555430}}] \label{lem:BS}
    Assume $f \in \R[x]_d$ is positive on $[-1,1]^n$ with minimum and maximum values $0 < f_\min <f_\max$ and that $R>1$.
    Then there exists an absolute constant $c \in [1,e^5]$, where $e$ is the base of the natural logarithm, such that
$f-\frac{1}{2}f_\min$ is nonnegative on $[-R,R]^n$ if
\[
R \le 1 + \frac{f_\min}{2cd^2f_\max}.
\]
\end{lemma}

We may now prove a weaker version of the Markov-Lucacs theorem (Theorem \ref{THM:MarkovLukacz}) that may readily be extended to the multivariate case.
\begin{theorem}
\label{thrm:main uni}
Assume $f \in \R[x]_d$ is positive on $[-1,1]$ with minimum and maximum values $0 < f_\min <f_\max$. Let $r \in \N$ be sufficiently large such that \eqref{eq:rbounds} holds as
well as
\begin{equation}
\label{eq:rbounds2}
\frac{r}{\log r}  \ge  150 d^{5/2}\frac{f_\max}{f_\min}.
\end{equation}
  Then we have
\[
f + \eps \in \mathcal Q(1-x^2)_{2t},
\]
where $t=O(r)$ and $\eps = O(\|f\|_\cheb \log(r)/r^2)$. In particular, we may assume
\[
t = {104r}, \mbox{ and }
\eps = \left( {\frac{7}{2}d^{9/2}} + 14\right)\|f\|_\cheb \log(r)/r^2.
\]
\end{theorem}

\begin{proof}
Let $r$ meet the conditions in the statement of the theorem,  and construct the associated convolution
 operator $\mathcal K_r$, where the parameters $\delta, \sigma, R$ are as in \cref{thrm:approxIDuni}.
  By \cref{lem:BS} and our assumption on $r$, we have that $f>0$ on $[-R,R]$.
  Therefore, by \cref{prop:Kr is SOS}, we have that ${\mathcal K_r f \in \Sigma[x]_t \subset \mathcal Q(1-x^2)_t}$ for  {$t = 104r$}.
   By \cref{thrm:approxIDuni} we moreover have $\|\mathcal K_r f-f\|_\cheb \le {\left( {\frac{7}{2}d^{9/2}} + 14\right)\frac{\log(r)}{r^2}}$.
   To conclude the proof, it thus suffices to apply \cref{lem:1approx}:
\[
f + \|\mathcal K_r f-f\|_\cheb \in \mathcal Q(1-x^2)_{2t}. \qedhere
\]
\end{proof}

\subsection{Extension to the multivariate setting}
\label{sec:multivariate proof}
Here we show how to extend the results from \cref{sec:proof uni} to the multivariate setting.

For the sake of readability, let us assume the following.
Let $d \in \N$ be fixed and $r \in \N$ with $r$ meets the conditions \eqref{eq:rbounds} and \eqref{eq:rbounds2}, and assume there exists a univariate kernel $K_r(x,y)$ for all $x \in [-1,1]$, $y \in [-R,R]$ for some $R>1$ with the following properties:
\begin{enumerate}
\item $\|\mathcal K_r T_k - T_k \|_\cheb \leq \eps$ for all $0 \leq k \leq d$.
\item $K_r(x,y) = s_r(x,y)^2$ for all $x \in [-1,1], y \in [-R,R]$, where $s_r \in \R[x]_t$ for $t = O(r)$.
\end{enumerate}

We now construct a multivariate kernel $K_{r,n}(x,y)$ for $x \in [-1,1]^n, y \in [-R,R]^n$ by setting
\[
K_{r,n}(x,y) = \prod_{i \in [n]} K_r(x_i,y_i).
\]
We first show that the first property of $\mathcal K_r$ implies that $\mathcal K$ approximately preserves multivariate Chebyshev polynomials $T_\alpha(x) = \prod_{i \in [n]} T_{\alpha_i}(x_i)$ in the following sense.
\begin{lemma} \label{thrm:approxIDmulti}
    Let $\alpha \in (\N_0)^n$ be such that $\alpha_i \leq d$ for all $i \in [n]$. Then we have
    \[
    \|\mathcal K_{r,n} T_\alpha - T_\alpha\|_\cheb \leq \eps \sum_{i=0}^{n-1} (1+\eps)^i.
    \]
    Moreover, if $\eps \leq 1/n$, then $\|\mathcal K_{r,n} T_\alpha - T_\alpha\|_\cheb \leq e \eps n$.
\end{lemma}
\begin{proof}
We prove the lemma using induction on $n$. The base case $n=1$ holds by our assumption on~$\mathcal K_r$. Now assume that $n>1$ and that the bound holds for up to $n-1$ variables.

We first observe that
\[
\mathcal K_{r,n} T_\alpha (x) = \int K_{r,n}(x,y) T_\alpha(y) dy = \prod_{i \in [n]} \left( \int K_r(x_i,y_i) T_{\alpha_i}(y_i) dy_i\right) = \prod_{i \in [n]} \mathcal K_r T_{\alpha_i}(x_i).
\]
Hence
\begin{align*}
\mathcal K_{r,n} T_\alpha - T_\alpha &= \prod_{i \in [n]} \mathcal K_r T_{\alpha_i} - \prod_{i \in [n]} T_{\alpha_i} \\
&= \prod_{i \in [n]} \mathcal K_r T_{\alpha_i} - \mathcal K_r T_{\alpha_n} \prod_{i \in [n-1]} T_{\alpha_i} + \mathcal K_r T_{\alpha_n} \prod_{i \in [n-1]} T_{\alpha_i} - \prod_{i \in [n]} T_{\alpha_i}\\
&= \mathcal K_r T_{\alpha_n} \left(\prod_{i \in [n-1]} \mathcal K_r T_{\alpha_i} - \prod_{i \in [n-1]} T_{\alpha_i}\right) + \left(\mathcal K_r T_{\alpha_n} -T_{\alpha_n} \right) \prod_{i \in [n-1]} T_{\alpha_i}.
\end{align*}
Using the triangle inequality and submultiplicativity of $\|\cdot\|_\cheb$ we thus obtain
\begin{align*}
    &\left\|\mathcal K_{r,n} T_\alpha - T_\alpha\right\|_\cheb \\
    &\leq  \left\| \mathcal K_r T_{\alpha_n} \left(\prod_{i \in [n-1]} \mathcal K_r T_{\alpha_i} - \prod_{i \in [n-1]} T_{\alpha_i}\right)\right\|_\cheb + \left\|\left(\mathcal K_r T_{\alpha_n} -T_{\alpha_n} \right) \prod_{i \in [n-1]} T_{\alpha_i}\right\|_\cheb \\
    &\leq \| \mathcal K_r T_{\alpha_n}\|_\cheb \left\|\prod_{i \in [n-1]} \mathcal K_r T_{\alpha_i} - \prod_{i \in [n-1]} T_{\alpha_i}\right\|_\cheb + \|\mathcal K_r T_{\alpha_n} -T_{\alpha_n}\|_\cheb \left\|\prod_{i \in [n-1]} T_{\alpha_i}\right\|_\cheb \\
    &\leq (1+\eps) \left\|\prod_{i \in [n-1]} \mathcal K_r T_{\alpha_i} - \prod_{i \in [n-1]} T_{\alpha_i}\right\|_\cheb + \eps,
\end{align*}
where in the last inequality we have used $\|\mathcal K_r T_{\alpha_n} - T_{\alpha_n}\|_\cheb \leq \eps$, $\|\mathcal K_r T_{\alpha_n}\|_\cheb \leq \|T_{\alpha_n}\|_\cheb + \|\mathcal K_r T_{\alpha_n} - T_{\alpha_n}\|_\cheb \leq 1+\eps$, and $\|\prod_{i \in [n-1]} T_{\alpha_i}\|_\cheb = 1$. We finally use the induction hypothesis to conclude
\[
\|\mathcal K_{r,n} T_\alpha - T_\alpha\|_\cheb \leq (1+\eps) \left(\eps \sum_{i=0}^{n-2} (1+\eps)^i\right) + \eps = \eps \sum_{i=0}^{n-1} (1+\eps)^i.
\]
{To prove the last statement of the lemma, note that, if $\eps \le 1/n$,
\[
\sum_{i=0}^{n-1} (1+\eps)^i < n\left(1+ \frac{1}{n}\right)^{n-1} < n\left(1+ \frac{1}{n}\right)^{n} < ne,
\]
where $e$ is the Euler number.}
\end{proof}

Next, we show that the second property of $\mathcal K_r$ implies that $\mathcal K_{r,n}$ maps positive polynomials to SOS polynomials.
 Using the same argument as in the proof of  \cref{prop:Kr is SOS}, we have the following.

\begin{proposition} \label{prop:prod Kr is SOS}
Assume $f \in \R[x]_d$ is nonnegative on $[-R,R]^n$, then $\mathcal{K}_{r,n} f$ is an SOS polynomial of degree $2nt = O(nr)$.
\end{proposition}
\begin{proof}
{
As in the proof of  \cref{prop:Kr is SOS}, we know --- by the Tchakaloff theorem \cite{Tchakaloff} --- that there exists a cubature rule with
nodes $\omega_1,\ldots,\omega_N \in [-R,R]^n$ and corresponding weights $c_1>0, \ldots, c_N >0$, for some $N \in \N$, such that
\[
\int_{[-R,R]^n} p(x)dx = \sum_{j=1}^N c_jp(\omega_j) \quad \mbox{for all } p \in \R[x]_{d + n\deg(K_r)}.
\]
Thus one has
\[
\mathcal K_{r,n} f(x) = \int_{[-R,R]^n} K_{r,n}(x,y) f(y) dy = \sum_{j=1}^N c_jK_{r,n}(x,\omega_j)f(\omega_j).
\]
The result now follows from the observation that $f(\omega_j) \ge 0$ and $x \mapsto K_{r,n}(x,\omega_j)$ is SOS of degree $n\deg(K_r) = 2nt$ for each $j \in \{1,\ldots,N\}$.
}
\end{proof}

We are now ready to prove our main result.

\begin{theorem}
\label{thrm:main multi}
Assume $f \in \R[x_1,\ldots,x_n]_d$ is positive on $[-1,1]^n$ with minimum and maximum values $0 < f_\min <f_\max$. Let $r \in \N$ be such that \eqref{eq:rbounds} and \eqref{eq:rbounds2} hold. Then we have
\[
f + \widetilde \eps  \in \mathcal Q(1-x_1^2,\ldots,1-x_n^2)_{2nt},
\]
where $t=O(r)$ and $\widetilde \eps = O(\|f\|_\cheb \log(r)/r^2)$.
In particular, we may assume
\[
t = {104r}, \mbox{ and }
\tilde\eps = e\cdot n\left( {\frac{7}{2}d^{9/2}} + 14\right)\|f\|_\cheb \log(r)/r^2.
\]
\end{theorem}
\begin{proof}
Let $r $ satisfy \eqref{eq:rbounds} and \eqref{eq:rbounds2}, and let the univariate kernel $K_r$ be as in \cref{thrm:approxIDuni}, and note that by our assumption on $r$ we have $\eps \coloneqq \max_{k \in \N: 0\leq k \leq d} \|\mathcal K_r T_k - T_k\|_\cheb \leq 1/n$. Let the parameters $\delta, \sigma, R$ be as in \cref{thrm:approxIDuni}.
By \cref{lem:BS} and our assumption on $r$, we have that $f>0$ on $[-R,R]^n$.
Therefore, by \cref{prop:prod Kr is SOS}, we have that ${\mathcal K_{r,n} f \in \Sigma[x]_{nt} \subset \mathcal Q(1-x_1^2,\ldots, 1-x_n^2)_{nt}}$
 for some $t = O(r)$. By \cref{thrm:approxIDmulti} we moreover have $\|\mathcal K_r f-f\|_\cheb \leq e\eps n \|f\|_\cheb = O(\|f\|_\cheb \log(r)/r^2)$. To conclude the proof, it thus suffices to apply \cref{lem:1approx}: we have
\[
f + \|\mathcal K_{r,n} f-f\|_\cheb \in \mathcal Q(1-x_1^2,\ldots,1-x_n^2)_{2nt}. \qedhere
\]
\end{proof}

We may now infer some results on the rate of convergence of the Lasserre hierarchy.

\begin{corollary}
\label{cor:main result}
Consider problem \eqref{prob:origin} of minimizing a polynomial $f$ of degree $d$ over $[-1,1]^n$, and the associated Lasserre hierarchy \eqref{eq:lasserre hierarchy}.
Fix $\varepsilon \in (0,1)$. If  $r \in \N$ satisfies
\[
\frac{r}{\log r} \ge 300\frac{d^{5/2}}{\varepsilon},
\]
then, for $t = {104r}$ one has
\[
f_\min - f_{(2nt)} \le \varepsilon(f_\max - f_\min) + e\cdot n\left( {\frac{7}{2}d^{9/2}} + 14\right)\|f\|_\cheb \log(r)/r^2.
\]
\end{corollary}
\begin{proof}
The result follows immediately by applying Theorem \ref{thrm:main multi} to the function
\[
\tilde f := f - f_\min + \varepsilon (f_\max - f_\min),
\]
and noting that
\[
\frac{\tilde f_\max}{\tilde f_\min} = \frac{1+ \varepsilon}{\varepsilon} < \frac{2}{\epsilon}.
\]
\end{proof}
A few remarks on our results:
\begin{itemize}
\item
   There are no constants in Corollary \ref{cor:main result} with explicit exponential dependence on $d$ and $n$,  unlike the $C(n,d)$ in Theorem \ref{thm:Baldi_Slot}. The quantity $\|f\|_\cheb$ may depend exponentially on $n$ and $d$, but it also may not in some interesting cases.
   For general $f \in \R[x]_d$, one has
\[
\|f\|_\cheb \le \left(2^d \binom{n+d}{d}\right)^{ \frac{1}{2}}\|f\|_{\infty},
\]
that follows from the (straightforward) multivariate extension of Lemma \ref{lem:normconversion}. On the other hand, for special cases like the maximum cut problem in graphs (as discussed in the introduction), one has $\|f\|_\cheb \le 2\|f\|_{\infty}$; see \cite{De_Klerk_Vera_2024} for a detailed discussion of this.

{A more effective dependence on $n$ and $d$ directly impacts the convergence rates of the moment-SOS hierarchy
 for the generalized moment problem, which is particularly relevant in applications such as optimal control
 (see, e.g., Schlosser et al.\ \cite{Schlosser_et_al}). On the other hand,
 our reliance on the specific coefficient norm within the Chebyshev polynomial basis presents challenges
 for extending the approach beyond polynomial optimization. When the function
 that needs to be certified as nonnegative is not a polynomial, but rather a converging series, the choice of norm can have nontrivial implications. }
    \item
   One should note that, for fixed $\varepsilon > 0$, the Lasserre
   hierarchy yields an $\varepsilon$-approximation of $f_\min$ after a \textit{finite} number of steps, since it converges to $f_\min$.
   For this reason,
   Corollary \ref{cor:main result} should not be seen as an asymptotic convergence result as $r \gg 0$, and we have therefore avoided big-O notation, i.e. we have made all constants explicit.

   \end{itemize}

\section{Limits to our proof technique}
\label{sec:limits}

Our approach is based on the best degree-$r$ SOS approximation of  $f- f_\min$ on $[-1,1]^n$ in the $\|\cdot\|_\cheb$ norm, as detailed in Corollary \ref{cor:1norm}.
 Here we argue that
an SOS $\delta$-approximation in the Chebyshev basis requires degree $\Omega(1/\sqrt{\delta})$.
{To do so, we leverage a well-known degree lower bound due to Stengle. We point out that the bound of Stengle relies on a non-standard description of the cube, but our degree lower bound in Theorem 10 does not.}
 Subsequently, we present a numerical investigation to get an indication whether this bound is tight.

\subsection{Bounds based on examples by Stengle}
We first recall degree lower- and upper bounds due to Stengle  
based on a non-standard algebraic description of the cube~\cite{STENGLE1996167}.

\begin{theorem}[Stengle, lower bound] \label{thrm:stenglelb}
    For any $\eps>0$ and $r \in \N$, we have
    \[
    (1-x^2) + \eps \in \mathcal Q((1-x^2)^3)_r \Rightarrow r = \Omega(1/\sqrt{\eps}).
    \]
\end{theorem}

\begin{theorem}[Stengle, upper bound] \label{thrm:stengleub}
    For any $\eps>0$, there exists an $r = O(\log(1/\eps)/\sqrt{\eps})$ such that
    \[
    (1-x^2) + \eps \in \mathcal Q((1-x^2)^3)_r.
    \]
\end{theorem}

Our goal is to prove the following.

\begin{theorem} \label{thrm:LBcheb}
If $q$ is a degree-$r$ sum of squares with $\|(1-x^2)-q\|_\cheb \leq \delta$, then we have $r = \Omega(1/\sqrt{\delta})$.
\end{theorem}

Before we prove the theorem, let us prove a useful lemma.

\begin{lemma} \label{lem:MonChebk}
If $R>0$ and $d \in \N$ are such that $R - x \in \mathcal Q((1-x^2)^3)_d$, then, for any $k \in \N$, we have
$R - x^k, R+x^k \in \mQ((1-x^2)^3)_{dk}$ and $R - T_k(x), R+T_k(x) \in \mQ((1-x^2)^3)_{dk}$.
\end{lemma}
\begin{proof}
Assume $R - x \in \mathcal Q((1-x^2)^3)_d$. Note that then also $R+x \in \mathcal Q((1-x^2)^3)_d$, which follows from $R+x = R-(-x)$ and the observation that $(1-x^2)^3 = (1-(-x)^2)^3$.
Then for $\epsilon \in \{-1,1\}$, we have
\begin{equation}\label{eq:Rx} R-\epsilon x = (1-x^2)^3 \alpha(x) + \beta(x),
\end{equation}
where $\deg \alpha \le d-6$, $\deg \beta \le d$ and $\alpha$ and $\beta$ are SOS.
Substituting $x \mapsto x^k$ in \eqref{eq:Rx} we obtain
\[R - \epsilon x^k = (1-x^{2k})^3\alpha(x^k) + \beta(x^k)
= (1-x^2)^3(1+x^2+ \cdots+ x^{2k-2})^3\alpha(x^k) + \beta(x^k) \in \mQ((1-x^2)^3)_{dk}.\]
This proves the lemma in the monomial case.

For the Chebyshev case, we use the cosine representation.
Substituting $x \mapsto \cos (kt)$ in \eqref{eq:Rx} we obtain,
\[R - \epsilon\cos(kt) = (1- \cos^2(kt))^3\alpha(\cos(kt)) + \beta(cos(kt)) = \sin^6(kt)\alpha(\cos(kt)) + \beta(\cos(kt)).
\]
By an inductive argument $\sin(kt) = \sin(t) p_k(\cos(t))$, where $p_k$ is a degree $k-1$ polynomial.
Then,
\[R - \epsilon \cos(kt) =  \sin^6(t)p_k(\cos(t))^6\alpha(\cos(kt)) + \beta(\cos(kt)).
\]
Equivalently,
\[R - \epsilon T_k(x) =  (1-x^2)^3p_k(x)^6\alpha(T_k(x)) + \beta(T_k(x)) \in \mQ((1-x^2)^3)_{dk}.
\]
This proves the lemma in the Chebyshev case.
\end{proof}

\begin{proof}[Proof of \cref{thrm:LBcheb}]
    We will set $R=2$. By \cref{thrm:stengleub} we have $3-x^2 = 1+2-x^2 \in \mathcal Q((1-x^2)^3)_d$ for some $d \in O(1)$. Using the identity $1+ \eps/2- x = \tfrac 12((1 - x)^2 + 1+\eps - x^2)$, for $\eps=1$, we obtain $R-x=2-x \in  \mQ((1-x^2)^3)_{d}$ for $d=O(1)$. 
    \cref{lem:MonChebk} then implies that for any $k \in \N$ we have $R - T_k(x),R+T_k(x) \in \mathcal Q((1-x^2)^3)_{dk}$.
    This yields the following variant of \cref{cor:1cheb}: for any polynomial $p$, we have $2\|p\|_\cheb +p \in \mathcal{Q}((1-x^2)^3)_{d\deg(p)}$. In particular, for $p = (1-x^2) -q$ we have
    \[
    2\|(1-x^2) -q\|_\mon + (1-x^2) -q \in \mathcal{Q}((1-x^2)^3)_{dr},
    \]
which, since $q$ is a sum of squares,  means
\[
(1-x^2) +2\|(1-x^2) -q\|_\cheb  \in \mathcal{Q}((1-x^2)^3)_{dr}.
\]
Note that $dr \in O(r)$ since $d \in O(1)$. We can now apply \cref{thrm:stenglelb} with $\eps = 2 \delta$ to obtain that $r = \Omega(1/\sqrt{\delta})$.
\end{proof}

\subsection{Numerical investigation}
In this section we do a numerical investigation to construct kernels that give SOS approximations in the $\|\cdot\|_\cheb$ norm, to give an indication
if it is possible to improve on our construction.
To this end, for fixed $d \in N$, we consider a family of kernels of the form
\begin{equation}\label{Kernel_Juan}
  K_r(x,y) := 1 + 2\sum_{k=1}^d \lambda_k T_k(x)T_k(y)  + 2\sum_{i,j = d+1}^r  \alpha_{ij} T_i(x)T_j(y)\quad \quad r > d,
\end{equation}
where the $\lambda_k \in [0,1]$ and $\alpha_{ij} \in [0,1]$ are fixed coefficients. As before, one has the associated approximation (convolution) operators
\[
\mathcal{K}_r(f)(x) = \int_{-1}^1 K_r(x,y) f(y) \frac{1}{\pi \sqrt{1-y^2}}\mathrm{d}y \quad \quad f \in \R[x], \; r \in \N.
\]
If $f \in \R[x]_d$, say $f(x) = \sum_{k=0}^d f_kT_k$, and $r > d$, one therefore has
\begin{eqnarray*}
  \|\mathcal{K}_r(f) - f\|_\cheb &=&  \left\|\sum_{k=1}^d (\lambda_k -1)f_kT_k\right\|_\cheb \\
   &=& \sum_{k=1}^d |(\lambda_k - 1)f_k| \\
   &\le &\max_{\ell \in [d]} |\lambda_\ell - 1|\cdot \|f\|_\cheb.
\end{eqnarray*}
In the spirit of the works \cite{doi:10.1137/22M1494476,De_Klerk_Vera_2024} one may now construct an optimal kernel via semidefinite programming
by solving, for fixed $d < r \in \N$:
\begin{equation}\label{def:v_rd}
  v_{r,d} := \min_{\lambda_k, \alpha_{ij} \in [0,1]} \max_{\ell \in [d]}\left\{|\lambda_\ell - 1| \; : \; K_r(x,y) \in \Sigma[x,y]_{2r} \mbox{ of the form } \eqref{Kernel_Juan}\right\}.
\end{equation}
By construction, one therefore has that if $f \in \R[x]_d$, then $\|\mathcal{K}_r(f) - f\|_\cheb \le v_{r,d}\cdot \|f\|_\cheb$. Moreover,
if $f$ is nonnegative in $[-1,1]$, then $\mathcal{K}_r(f)$ is SOS.
In Table \ref{tab:vrd} we give the numerical values of $v_{r,d}$ for some (small) values of $r$ and $d$.

\begin{table}[h!]
\centering
{\footnotesize
\begin{tabular}{ ccccccccccccc}
 & $d=$ 2 & 3 & 4 & 5 & 6 & 7 & 8 & 9 & 10 & 11 \\
 \hline
$r =$ 1 & 0.9999 \\
2 & 0.9954 & 0.9998 \\
3 & 0.9641 & 0.9958 & 0.9998 \\
4 & 0.8993 & 0.9813 & 0.9975 & 0.9998 \\
5 & 0.8116 & 0.9475 & 0.9892 & 0.9973 & 0.9995 \\
6 & 0.7272 & 0.9062 & 0.9733 & 0.9914 & 0.9974 & 0.9997 \\
7 & 0.6624 & 0.8552 & 0.9484 & 0.9798 & 0.9940 & 0.9979 & 0.9996 \\
8 & 0.5515 & 0.7932 & 0.9096 & 0.9657 & 0.9870 & 0.9956 & 0.9982 & 0.9994 \\
9 & 0.4899 & 0.7383 & 0.8796 & 0.9455 & 0.9760 & 0.9887 & 0.9955 & 0.9982 & 0.9995 \\
10 & 0.4037 & 0.6798 & 0.8441 & 0.9104 & 0.9595 & 0.9834 & 0.9922 & 0.9966 & 0.9983 & 0.9994 \\
11 & 0.3519 & 0.6173 & 0.7958 & 0.8949 & 0.9450 & 0.9738 & 0.9866 & 0.9938 & 0.9968 & 0.9987  \\
12 & 0.3097 & 0.5676 & 0.7502 & 0.8614 \\
\hline
\end{tabular}
}
\caption{The values $v_{r,d}$ from \eqref{def:v_rd} for $r \le 12$ and  $d \le 11$. \label{tab:vrd}}
\end{table}

\begin{figure}[ht!]
\begin{center}

\begin{tikzpicture}[scale=1.1]
  \begin{axis}[
      xlabel={$r$},
      ylabel={$1/v_{r,d}$},
      legend pos=outer north east,
      grid=both,
      xmin=1, xmax=12,
      xtick={1,2,...,12},
  ]
    \addplot+[mark=square*] coordinates {
      (2,1.0001) (3,1.0046) (4,1.0372) (5,1.1119)
      (6,1.2321) (7,1.3752) (8,1.5097) (9,1.8132)
      (10,2.0411) (11,2.4770) (12,3.2286)
    };
    \addlegendentry{$d=2$}

    \addplot+[mark=triangle*] coordinates {
      (3,1.0002) (4,1.0042) (5,1.0191) (6,1.0554)
      (7,1.1035) (8,1.1693) (9,1.2607) (10,1.3545)
      (11,1.4710) (12,1.6200)
    };
    \addlegendentry{$d=3$}

    \addplot+[mark=diamond*] coordinates {
      (4,1.0002) (5,1.0025) (6,1.0109) (7,1.0274)
      (8,1.0544) (9,1.0994) (10,1.1369) (11,1.1847)
      (12,1.2566)
    };
    \addlegendentry{$d=4$}

    \addplot+[mark=star] coordinates {
      (5,1.0002) (6,1.0027) (7,1.0087) (8,1.0207)
      (9,1.0356) (10,1.0577) (11,1.0984) (12,1.1175)
    };
    \addlegendentry{$d=5$}
  \end{axis}
\end{tikzpicture}
\end{center}
\caption{Plots of the values $1/v_{r,d}$ as a function of $r$ for fixed  $d \in \{1,2,3,4\}$. \label{fig:vrd1}}
\end{figure}

\begin{figure}[ht!]
\begin{center}
\begin{tikzpicture}[scale=1.1]
  \begin{axis}[
      xlabel={$r$},
      ylabel={$v_{r,d}\,\dfrac{r}{d}$},
      legend pos=outer north east,
      grid=both,
      xmin=1, xmax=12,
      xtick={1,2,...,12},
  ]
    \addplot+[mark=o] coordinates {
      (1,0.9999) (2,1.9908) (3,2.8923) (4,3.5972)
      (5,4.0580) (6,4.3632) (7,4.6368) (8,4.4120)
      (9,4.4091) (10,4.0370) (11,3.8709) (12,3.7164)
    };
    \addlegendentry{$d=1$}

    \addplot+[mark=square*] coordinates {
      (2,0.9998) (3,1.4937) (4,1.9626) (5,2.3688)
      (6,2.7186) (7,2.9932) (8,3.1728) (9,3.3224)
      (10,3.3990) (11,3.3952) (12,3.4056)
    };
    \addlegendentry{$d=2$}

    \addplot+[mark=triangle*] coordinates {
      (3,0.9998) (4,1.3307) (5,1.6487) (6,1.9466)
      (7,2.2129) (8,2.4264) (9,2.6388) (10,2.8137)
      (11,2.9179) (12,3.0008)
    };
    \addlegendentry{$d=3$}

    \addplot+[mark=diamond*] coordinates {
      (4,0.9998) (5,1.2466) (6,1.4600) (7,1.7146)
      (8,1.9314) (9,2.1274) (10,2.2760) (11,2.4610)
      (12,2.5842)
    };
    \addlegendentry{$d=4$}
  \end{axis}
\end{tikzpicture}

\end{center}
\caption{Plots of the values $v_{r,d}\cdot (r/d)$ as a function of $r$ for fixed  $d \in \{1,2,3,4\}$. \label{fig:vrd2}}
\end{figure}

\begin{figure}[ht!]
\begin{center}

\begin{tikzpicture}[scale=1.1]
  \begin{axis}[
      xlabel={$r$},
      ylabel={$v_{r,d}\,\bigl(\tfrac{r}{d}\bigr)^{2}$},
      legend pos=outer north east,
      grid=both,
      xmin=1, xmax=12,
      xtick={1,2,...,12},
  ]
    \addplot+[mark=o] coordinates {
      (1,0.9999) (2,3.9816) (3,8.6769) (4,14.3888)
      (5,20.2900) (6,26.1792) (7,32.4576) (8,35.2960)
      (9,39.6819) (10,40.3700) (11,42.5799) (12,44.5968)
    };
    \addlegendentry{$d=1$}

    \addplot+[mark=square*] coordinates {
      (2,0.9998) (3,2.2406) (4,3.9252) (5,5.9219)
      (6,8.1558) (7,10.4762) (8,12.6912) (9,14.9506)
      (10,16.9950) (11,18.6733) (12,20.4336)
    };
    \addlegendentry{$d=2$}

    \addplot+[mark=triangle*] coordinates {
      (3,0.9998) (4,1.7733) (5,2.7478) (6,3.8932)
      (7,5.1635) (8,6.4683) (9,7.9164) (10,9.3789)
      (11,10.6991) (12,12.0032)
    };
    \addlegendentry{$d=3$}

    \addplot+[mark=diamond*] coordinates {
      (4,0.9998) (5,1.5583) (6,2.2306) (7,3.0006)
      (8,3.8628) (9,4.7866) (10,5.6900) (11,6.7677)
      (12,7.7526)
    };
    \addlegendentry{$d=4$}
  \end{axis}
\end{tikzpicture}
\end{center}
\caption{Plots of the values $v_{r,d}\cdot (r/d)^2$ as a function of $r$ for fixed  $d \in \{1,2,3,4\}$. \label{fig:vrd3}}
\end{figure}

Some of the results in Table \ref{tab:vrd} are plotted in Figures \ref{fig:vrd1}, \ref{fig:vrd2}, and \ref{fig:vrd3}.
In Figure \ref{fig:vrd2}, one sees that the values $v_{r,d}\cdot (r/d)$ seem to be bounded from above for fixed $d$ and growing $r$, but this does not seem to be the case for the values
$v_{r,d}\cdot (r/d)^2$ in Figure \ref{fig:vrd3}.
On the other hand, the plot of the values $1/v_{r,d}$ in Figure \ref{fig:vrd1} shows a nonlinear increasing trend.
This would suggest that $v_{r,d} = o(d/r)$, which in turn would imply that the Lasserre hierarchy converges at a rate $o(1/r)$, which would be an improvement on the $O(1/r)$ result obtained by Baldi and Slot \cite{doi:10.1137/23M1555430}.
If one could find the analytical values of the kernel coefficients corresponding to $v_{r,d}$, then one could perhaps construct a proof showing $o(1/r)$ convergence; unfortunately, we have not been able to extract analytical expressions from the numerical solutions.

We should also point out that kernels of the form \eqref{Kernel_Juan} are special, since the convolution operator has the Chebyshev basis as eigenvectors. It may still be possible to show $O(1/r^2)$ SOS approximations by more general kernels.

\section{Conclusion and discussion}
\label{sec:conclusion}
It remains an open question as to what the exact rate of convergence of the Lasserre hierarchy is for problem \eqref{prob:origin}.
As we recalled in Corollary \ref{cor:Baldi_Slot}, Baldi and Slot \cite{doi:10.1137/23M1555430} showed a convergence rate of $O(1/r)$, but in the same
paper they also showed a lower bound of $\Omega(1/r^8)$.
There is therefore still quite some gap between the known lower and upper bounds.

There is some reason to expect the exact rate to be $O(1/r^2)$. Indeed, this rate has been shown by Laurent and Slot \cite{Laurent2021AnEV} (see also \cite{De_Klerk_Vera_2024}) for the related SDP
hierarchy where the quadratic module \eqref{eq:quad mod}  is replaced by the pre-ordering
\[
\mathcal{T}(\mathbf{g}) = \Sigma[x] + \sum_{\mathrm{I} \subseteq [n]} \Sigma[x] \prod_{i \in \mathrm{I}}g_i,
\]
which is a superset of the quadratic module \eqref{eq:quad mod}. SDP hierarchies of this type are sometimes called Schm\"{u}dgen hierarchies,
with reference to the celebrated Schm\"{u}dgen positivstellensatz \cite{Schmudgen1991}, which states that --- if one replaces the quadratic module in the statement of
Theorem \ref{thm:Putinar} by the pre-ordering --- then one may omit the Archimedean assumption. (It is worth noting that Baldi and Slot \cite{doi:10.1137/23M1555430} leveraged the results by Laurent and Slot \cite{Laurent2021AnEV} in their proof, whereas we have not used this analysis in this paper.)

Similarly, Bach and Rudi \cite{doi:10.1137/22M1540818} also obtained the same $O(1/r^2)$ rate for related SDP hierarchies to minimize trigonometric polynomials.
Moreover, Fang and Fawzi \cite{10.1007/s10107-020-01537-7} obtained the same convergence rate for the Lasserre hierarchy
{to minimize {\em homogeneous} polynomials} on a sphere (as opposed to the hypercube).
However, in none of these cases is the $O(1/r^2)$ rate known to be tight.

In this paper, we showed that there is still hope to use our proof technique to obtain the $O(1/r^2)$ rate for the
Lasserre hierarchy for problem \eqref{prob:origin}, but also showed that this would also be the limit for our proof technique. The most promising avenue to pursue would be
to obtain analytical expressions for the kernels \eqref{def:v_rd} that we investigated numerically.

An aspect of our proof strategy that may be of independent interest is the constructive approximation of nonnegative polynomials on the hypercube by SOS polynomials. This topic is of independent interest and has been studied by Lasserre and Netzer \cite{Lasserre_Netzer} and Lasserre \cite{doi:10.1137/04061413X}.

\bibliographystyle{plain}

\begin{thebibliography}{10}

\bibitem{doi:10.1137/22M1540818}
Francis Bach and Alessandro Rudi.
\newblock Exponential convergence of sum-of-squares hierarchies for
  trigonometric polynomials.
\newblock {\em SIAM Journal on Optimization}, 33(3):2137--2159, 2023.

\bibitem{doi:10.1137/23M1555430}
Lorenzo Baldi and Lucas Slot.
\newblock Degree bounds for {P}utinar’s {P}ositivstellensatz on the
  hypercube.
\newblock {\em SIAM Journal on Applied Algebra and Geometry}, 8(1):1--25, 2024.

\bibitem{10.5555/3226669.3227146}
Mihir Bellare and Phillip Rogaway.
\newblock The complexity of approximating a nonlinear program.
\newblock {\em Math. Program.}, 69(1–3):429--441, July 1995.

{
\bibitem{Chernoff}
Herman Chernoff. \newblock A measure of asymptotic efficiency for tests of a hypothesis based on the sum of observations.
\newblock  \emph{Ann. Math. Statist.} 23(4): 493--507,  1952.
}


\bibitem{doi:10.1137/100814147}
Etienne de~Klerk and Monique Laurent.
\newblock On the {L}asserre hierarchy of semidefinite programming relaxations
  of convex polynomial optimization problems.
\newblock {\em SIAM Journal on Optimization}, 21(3):824--832, 2011.

\bibitem{De_Klerk_Vera_2024}
Etienne de~Klerk and Juan~C. Vera.
\newblock The link between $ 1 $-norm approximation and effective
  positivstellensätze for the hypercube.
\newblock {\em Numerical Algebra, Control and Optimization}, 2024.

\bibitem{10.1007/s10107-020-01537-7}
Kun Fang and Hamza Fawzi.
\newblock The sum-of-squares hierarchy on the sphere and applications in
  quantum information theory.
\newblock {\em Math. Program.}, 190(1–2):331–360, November 2021.

\bibitem{10.1145/227683.227684}
Michel~X. Goemans and David~P. Williamson.
\newblock Improved approximation algorithms for maximum cut and satisfiability
  problems using semidefinite programming.
\newblock {\em J. ACM}, 42(6):1115--1145, November 1995.

\bibitem{10.1145/3597066.3597075}
Sander Gribling, Sven Polak, and Lucas Slot.
\newblock A note on the computational complexity of the moment-sos hierarchy
  for polynomial optimization.
\newblock In {\em Proceedings of the 2023 International Symposium on Symbolic
  and Algebraic Computation}, ISSAC '23, pages 280--288, New York, NY, USA,
  2023. Association for Computing Machinery.

{
\bibitem{Kotz}
Norman Lloyd Johnson, Samuel Kotz, and N. Balakrishnan. \newblock  {\em Continuous Univariate Distributions. 2nd ed.} \newblock  New York: Wiley, 1994.
}

\bibitem{doi:10.1137/22M1494476}
Felix Kirschner and Etienne de~Klerk.
\newblock Construction of multivariate polynomial approximation kernels via
  semidefinite programming.
\newblock {\em SIAM Journal on Optimization}, 33(2):513--537, 2023.

\bibitem{doi:10.1137/S1052623400366802}
Jean~B. Lasserre.
\newblock Global optimization with polynomials and the problem of moments.
\newblock {\em SIAM Journal on Optimization}, 11(3):796--817, 2001.

\bibitem{doi:10.1137/04061413X}
Jean~B. Lasserre.
\newblock A sum of squares approximation of nonnegative polynomials.
\newblock {\em SIAM Journal on Optimization}, 16(3):751--765, 2006.

\bibitem{doi:10.1137/080728214}
Jean~B. Lasserre.
\newblock Convexity in semi-algebraic geometry and polynomial optimization.
\newblock {\em SIAM Journal on Optimization}, 19(4):1995--2014, 2009.

\bibitem{Lasserre_Netzer}
Jean~B. {Lasserre} and Tim {Netzer}.
\newblock {SOS approximations of nonnegative polynomials via simple high degree
  perturbations}.
\newblock {\em Mathematische Zeitschrift}, 256:99--112, 2007.

\bibitem{Laurent2021AnEV}
Monique Laurent and Lucas Slot.
\newblock An effective version of {S}chm{\"u}dgen’s {P}ositivstellensatz for
  the hypercube.
\newblock {\em Optimization Letters}, 17:515--530, 2021.

\bibitem{Nie2012OptimalityCA}
Jiawang Nie.
\newblock Optimality conditions and finite convergence of {L}asserre's
  hierarchy.
\newblock {\em Mathematical Programming}, 146:97 -- 121, 2012.

\bibitem{NIE2007135}
Jiawang Nie and Markus Schweighofer.
\newblock On the complexity of putinar's positivstellensatz.
\newblock {\em Journal of Complexity}, 23(1):135--150, 2007.

\bibitem{3f4f597f-caa0-3104-a850-33b8de850a7c}
Victoria Powers and Bruce Reznick.
\newblock Polynomials that are positive on an interval.
\newblock {\em Transactions of the American Mathematical Society},
  352(10):4677--4692, 2000.

\bibitem{Putinar1993}
Mihai Putinar.
\newblock Positive polynomials on compact semi-algebraic sets.
\newblock {\em Indiana University Mathematics Journal}, 42(3):969--984, 1993.

\bibitem{rivlin2020chebyshev}
Theodore~J. Rivlin.
\newblock {\em Chebyshev Polynomials: From Approximation Theory to Algebra and
  Number Theory: Second Edition}.
\newblock Dover Books on Mathematics. Dover Publications, 2020.

\bibitem{sachdeva2014approx}
Sushant Sachdeva and Nisheeth~K. Vishnoi.
\newblock Faster algorithms via approximation theory.
\newblock {\em Foundations and Trends® in Theoretical Computer Science},
  9(2):125--210, 2014.


{\bibitem{Schlosser_et_al}
Corbinian Schlosser, Matteo Tacchi-B\'enard, and Alexey Lazarev. \newblock Convergence rates for the moment-SoS hierarchy.
\newblock \emph{Numerical Algebra, Control and Optimization}, to appear. \url{doi: 10.3934/naco.2025011}
}

\bibitem{Schmudgen1991}
Konrad Schm\"udgen.
\newblock The {K}-moment problem for compact semi-algebraic sets.
\newblock {\em Mathematische Annalen}, 289(2):203--206, 1991.

\bibitem{STENGLE1996167}
Gilbert Stengle.
\newblock Complexity estimates for the {S}chm\"udgen {P}ositivstellensatz.
\newblock {\em Journal of Complexity}, 12(2):167--174, 1996.

{
\bibitem{Tchakaloff}
Ljubomir Tchakaloff. Formules de cubature {m}\'{e}canique  \`{a} coefficients non n\'{e}gatifs, \emph{Bull. Sci. Math.},
81, 123--134 (1957)
}

\bibitem{RevModPhys.78.275}
Alexander Wei\ss{}e, Gerhard Wellein, Andreas Alvermann, and Holger Fehske.
\newblock The kernel polynomial method.
\newblock {\em Rev. Mod. Phys.}, 78:275--306, Mar 2006.

\end{thebibliography}

\end{document}